\documentclass[11pt]{article}
\usepackage[T1]{fontenc}
\usepackage[utf8]{inputenc}
\usepackage{lmodern}
\usepackage[margin=1in]{geometry}
\usepackage{amsmath}
\usepackage{amssymb}
\usepackage{amsthm}
\usepackage{bm}
\usepackage{graphicx}
\usepackage{booktabs}
\usepackage{array}
\usepackage{longtable}
\usepackage{algorithm}
\usepackage{algpseudocode}
\usepackage{tikz}
\usetikzlibrary{arrows.meta,positioning}
\usepackage[round]{natbib}
\usepackage{hyperref}

\theoremstyle{plain}
\newtheorem{theorem}{Theorem}
\newtheorem{lemma}{Lemma}
\theoremstyle{remark}
\newtheorem{remark}{Remark}
\theoremstyle{definition}
\newtheorem{counterexample}{Counterexample}

\newcommand{\figcaption}[1]{\caption{#1}}

\title{Admission and Assortment Optimization for Multi-size Automated Parcel Lockers}

\author{%
  Carlos An\'{\i}bal Su\'arez\thanks{Facultad de Ciencias Naturales y Matem\'aticas, Escuela Superior Polit\'{e}cnica del Litoral, Guayaquil, Ecuador; \href{mailto:carasuar@espol.edu.ec}{carasuar@espol.edu.ec}} \and
  Antoine Deza\thanks{Department of Computing and Software, Faculty of Engineering, McMaster University, Hamilton, Ontario, Canada; \href{mailto:deza@mcmaster.ca}{deza@mcmaster.ca}} \and
  Tal Raviv\thanks{School of Industrial \& Intelligent Systems Engineering, Faculty of Engineering, Tel Aviv University, Israel; \href{mailto:talraviv@tau.ac.il}{talraviv@tau.ac.il}. Corresponding author.}%
}
\date{\today}

\begin{document}
\maketitle

\begin{abstract}
We study admission control and capacity design for automated parcel lockers with multiple parcel and locker sizes. A smaller parcel can use a larger locker, but doing so may block a future larger parcel whose rejection is more costly. We formulate the admission problem as a finite-state, infinite-horizon average-cost Markov decision process and solve small instances exactly by relative value iteration. We analyze the \emph{always-accept} (AA) policy, which admits every feasible parcel into the remaining compatible capacity, and give a sufficient condition for its optimality. Across two-, three-, and four-size experiments, AA is optimal in fast-pickup regimes and nearly optimal when holding times are longer; observed optimality gaps are negligible even when AA is not optimal. We then study the locker-assortment problem, which minimizes facility cost plus optimal expected rejection cost. We give an exact bound-and-enumerate algorithm for moderate-size instances. Although the objective is not discrete convex, exchange-neighborhood local search finds the certified optimum in every instance for which exact certification is computationally tractable, and it scales as a heuristic to larger systems.
\end{abstract}


\medskip

\section{Introduction}

Automated parcel lockers are a self-service pickup system used in last-mile delivery where a courier delivers parcels to a finite bank of lockers; customers receive a notification and access code, and each parcel occupies capacity until it is collected or returned after a pickup window. Empirical studies and industry implementations highlight both their operational value and the associated control challenge. Lockers consolidate deliveries and reduce failed home-delivery attempts \citep{IwanKijewskaLemke2016,RanjbariEtAl2023DeliveryTimes}, but capacity is scarce, customer pickup times are uncertain, and poor admission or reservation policies can reject profitable requests even when compatible capacity exists \citep{SethuramanEtAl2024}.
A Seattle residential deployment studied by \citet{RanjbariEtAl2023DeliveryTimes} used 55 lockers, split into 19 small, 28 medium, and 8 large lockers. In the Amazon Locker setting, \citet{SethuramanEtAl2024} consider 100-slot lockers as a planning example and report that capacity management increased locker throughput by 9\%.
This paper studies the corresponding admission-control question: when parcels of different sizes compete for lockers with different capacities, should the operator reserve larger lockers for future larger parcels, or admit every feasible parcel as it arrives?

The analytical framework is a finite- or infinite-horizon Markov decision process. Results in dynamic programming establish the existence of optimal stationary policies in finite-state average-cost models and provide the Bellman-equation foundation used in this paper \citep{Puterman1994}. The literature on monotone comparative-statics shows that convexity and supermodularity often imply threshold or monotone policies, both in static optimization and in dynamic control \citep{Topkis1998}. These ideas arise in stochastic inventory theory, where convex value functions and one-dimensional marginal comparisons lead to structured policies such as base-stock and reorder-point policies \citep{Porteus2002}. 

The decision of whether to place an additional small parcel in a large locker is related to dynamic admission and stochastic knapsack problems. In these models, requests arrive over time, consume scarce capacity, and must be accepted or rejected based on their current reward and their impact on future opportunities. \citet{PapastavrouRajagopalanKleywegt1996} and \citet{KleywegtPapastavrou1998,KleywegtPapastavrou2001} show that dynamic stochastic knapsack formulations often admit threshold acceptance policies and exhibit convexity and monotonicity properties. Our setting differs because capacity is partitioned into several locker sizes and because accepted small parcels may later migrate from large lockers to small ones. Admitting a small parcel into flexible capacity generates an immediate gain, but also an opportunity cost through the increased risk of rejecting future high-priority demand.
 
\citet{GurvichPerry2012} study overflow networks motivated by call-center outsourcing and show how blocking and overflow interactions can generate tractable approximations. Our problem differs in its discrete-capacity, parcel-storage context, but similar questions arise when using flexible capacity for one size changes the future loss exposure of another size. The large lockers act as both a dedicated capacity for large parcels and an overflow buffer for small parcels.

Research on parcel lockers has expanded rapidly, especially in response to the growth of e-commerce and alternatives to home delivery. Reviews by \citet{MaWongTeo2022} indicate a focus on network design, facility location, capacity sizing, layout, routing, sustainability, and consumer adoption. Empirical work by \citet{MaTeoWong2024} uses parcel-locker network data to analyze how spatial characteristics affect demand and operational performance. Network-level models study how service points should be embedded in last-mile systems. \citet{OrensteinRaviv2022} introduce a hyperconnected service network in which service points can serve as pickup, drop-off, and intermediate storage nodes, and develop routing and scheduling tools for such networks. Strategic design models, such as \citet{DeutschGolany2018}, \citet{raviv2023service}, and \citet{SabzevariZadehEtAl2026}, study where lockers should be located and how much capacity should be installed. 

\citet{raviv2023service} emphasizes the interaction between local congestion and network-level service-point location and capacity decisions in a single-size setting, while \citet{SabzevariZadehEtAl2026} integrate parcel-locker location, capacity, fleet, replenishment, and routing decisions in a multi-period urban delivery model. At the local design level, \citet{FaugereMontreuil2020} optimize smart locker-bank configurations and compare monolithic and modular designs, while \citet{Kahr2022} integrates locker location and layout decisions with commodity-specific locker sizes. In Kahr's computational study, parcels and sizes are represented by small, medium, large, and extra-large sizes. However, the capacity constraints are commodity-specific and do not model the substitutability among locker sizes that is central to the present paper. In Kahr's model, locker sizes are mutually exclusive resources; in ours, larger lockers can serve as overflow capacity for smaller parcels.

A subset of the locker literature addresses the micro-level operational questions most relevant to this paper. \citet{OrensteinRavivSadan2019} study flexible delivery to automated parcel lockers when recipients may accept several candidate service points, showing how flexibility of delivery can improve system performance. 
They consider parcels and lockers of different sizes, but their models are deterministic and myopic rather than stochastic and dynamic. 
In the Amazon Locker setting, \citet{SethuramanEtAl2024} develop a reservation system that protects locker space for higher-priority shipping options using demand forecasting and optimization.
Their work shows that first-come-first-served use of locker capacity can be substantially improved by explicitly reserving flexible capacity for more valuable future demand. 
Their setting is single-size with multiple shipping priorities; ours is multi-size with a single priority class. The two problems lead to different operational recommendations -- reservation across priorities is valuable in Sethuraman et al., whereas reservation across sizes is rarely beneficial in our tested regimes, as shown in Sections~\ref{sec:policy-analysis}--\ref{sec:numerical-evidence}.

\citet{Mancini2026} consider a dynamic stochastic parcel-locker assignment problem with uncertain pickup times, in which delivery orders must be accepted or rejected online and then assigned to lockers. However, they also assume parcels of uniform size.
 
Our contributions are a multi-size MDP formulation, a sufficient optimality condition for AA, numerical evidence that AA is optimal or economically indistinguishable from optimal in the tested regimes, an exact bound-and-enumerate method for locker assortment optimization, and a scalable local-search heuristic for large service points. 

The paper is organized as follows. Section~\ref{sec:problem-description} defines the multi-size locker-admission problem and provides an exact relative value iteration algorithm for the general \(n\)-size model. Section~\ref{sec:policy-analysis} analyzes the AA policy, Section~\ref{sec:numerical-evidence} reports numerical evidence on AA performance, and Section~\ref{sec:locker-assortment} studies the locker-assortment problem. The supplementary material reports pickup-time robustness checks, computational reproducibility details, the assortment demand grid, and a notation table.

\section{Markov Decision Process (MDP) for the Multi-size Locker-Admission Problem}\label{sec:problem-description}

\subsection{Multi-size Locker-Admission Problem}\label{sec:problem-definition}
Consider an automated parcel locker service point with $n$ locker sizes indexed by \(1,2,\dots,n\) in increasing order of size. Let $b_i$ be the number of lockers of size $i$. Parcels also have sizes \(1,\dots,n\): a size-$i$ parcel can fit in any locker of size $i,i+1,\dots,n$, but not in a smaller locker. Let \(X_t=(X_t^1,\dots,X_t^n)\) be the arrival vector in cycle $t$, and assume that \(\{X_t\}_{t\ge 1}\) is i.i.d.\ across cycles. The MDP formulation allows any finite-support or truncated arrival distribution. In the numerical and assortment experiments, arrivals are independent Poisson random variables with means \(\lambda_i\). Rejection penalties satisfy \(0<\pi_1<\pi_2<\cdots<\pi_n\).

Each replenishment cycle has two phases before customer pickups. First, the parcels already present in the service point are scanned in increasing order of size and moved, whenever possible, to the smallest feasible locker. This relocation step is deterministic and does not involve an admission-control decision; it abstracts from any courier time required to move parcels within the bank. Second, new arrivals are considered in decreasing order of size, from $n$ down to $1$. At this stage, the decision maker chooses how many arriving parcels of each size to admit, subject to the multi-size compatibility constraints and the lockers that remain available after larger-size decisions have been made. An admitted size-$i$ parcel can be placed in any available locker of size $i,i+1,\dots,n$, but the model does not require all feasible spillover to be used. Any arrival that is not admitted is rejected and incurs its size-specific penalty.

After replenishment, pickups occur during the day. We assume a common memoryless pickup process: every parcel currently in the system departs independently with probability $p$, regardless of its size, seniority, or locker type. Supplementary material tests the sensitivity of this simplifying assumption. See Table~\ref{tab:core-notation} for the notation used in the model.

\begin{table}[htbp]
\centering
\caption{Notation for the multi-size admission model}
\label{tab:core-notation}
\small
\begin{tabular}{@{}ll@{}}
\toprule
Symbol & Meaning \\
\midrule
\(n\) & Number of parcel and locker sizes, ordered from smallest to largest \\
\(b_i\), \(B_i=\sum_{k=i}^n b_k\), \(C_i=\sum_{k=1}^i b_k\) & Size-\(i\) capacity, tail capacity, and prefix capacity \\
\(X_t\), \(A\), \(a\) & Arrival vector in cycle \(t\), generic arrival vector, and realized arrival vector \\
\(\lambda_i\) & Mean size-\(i\) Poisson arrival rate in the numerical experiments \\
\(x\in\mathcal S_n\) & Before-replenishment state vector \\
\(y\in\mathcal U(x,a)\) & Feasible post-replenishment parcel-count vector after arrivals \(a\) \\
\(D\) & Pickup vector, with \(D_i\mid y_i\sim \mathrm{Bin}(y_i,p)\) independently \\
\(\pi_i\) & Rejection penalty for a size-\(i\) parcel \\
\(g\), \(h\) & Average cost and relative value function in the Bellman equation \\
\bottomrule
\end{tabular}
\normalsize
\end{table}

\subsection{State Space and Feasible Decisions}\label{sec:nsize-model}
Given a relocation policy, the exact placement pattern of the parcels before replenishment is not needed in the state. It is enough to record \(x=(x_1,\dots,x_n)\), where $x_i$ is the number of size-$i$ parcels currently in the system. Define the tail capacities \(B_i=\sum_{k=i}^n b_k\), \(i=1,\dots,n\), and the prefix capacities \(C_i=\sum_{k=1}^i b_k\), with \(C_0=0\).
The feasible state space is
\[
\mathcal S_n=
\left\{
x\in\mathbb Z_+^n:
\sum_{k=i}^n x_k \le B_i,\quad i=1,\dots,n
\right\},
\]
because parcels of sizes $i,i+1,\dots,n$ cannot use lockers smaller than size $i$.

Let \(A=(A_1,\dots,A_n)\) denote a generic copy of the arrival vector and write \(a=(a_1,\dots,a_n)\) for a realization. Given state \(x\) and arrivals \(a\), let \(y=(y_1,\dots,y_n)\) denote the post-replenishment parcel-count vector. The physically feasible
post-decision set is
\[
\mathcal Y(x,a)=
\left\{
y\in\mathbb Z_+^n:
x_i \le y_i \le x_i+a_i,\quad
\sum_{k=i}^n y_k \le B_i,\quad i=1,\dots,n
\right\}.
\]
The MDP action is not every vector in \(\mathcal Y(x,a)\), because arrivals are
processed in decreasing size order, and a parcel that fits into an available
locker of its own size is not rejected. To formalize this, define recursively
\(s_0(x)=0\) and
\[
s_i(x)=\bigl(s_{i-1}(x)+x_i-b_i\bigr)^+,\qquad i=1,\dots,n,
\]
where \(s_i(x)\) is the number of parcels of sizes \(1,\dots,i\) that must
occupy lockers larger than size \(i\) after the relocation step. Thus
\(\ell_i(x)=\min\{b_i,s_{i-1}(x)\}\) is the number of size-\(<i\) parcels
occupying size-\(i\) lockers after relocation, and
\(\phi_i(x)=(b_i-\ell_i(x)-x_i)^+\) is the number of size-\(i\) lockers still
available after relocating existing size-\(<i\) and size-\(i\) parcels. The admissible action space is
\[
\mathcal U(x,a)=
\left\{
y\in\mathcal Y(x,a):
y_i\ge x_i+\min\{a_i,\phi_i(x)\},\quad i=1,\dots,n
\right\}.
\]
Equivalently, the action can be represented by an admission vector
\(u=y-x\), where \(0\le u_i\le a_i\) is the number of admitted size-\(i\) arrivals and at least \(\min\{a_i,\phi_i(x)\}\) of them must be admitted to
size-\(i\) lockers.  Representing the action by \(y\) makes the transition law immediate.
A stationary admission policy is a mapping
\(\mu:\mathcal S_n\times\mathbb Z_+^n\to\mathcal S_n\) satisfying
\(\mu(x,a)\in\mathcal U(x,a)\) for every feasible state \(x\) and arrival
realization \(a\).
The one-period rejection cost is \(c(x,a,y)=\sum_{i=1}^n \pi_i(x_i+a_i-y_i)\). Conditional on \(y\), let \(D=(D_1,\dots,D_n)\) be the pickup vector, with independent components \(D_i\mid y_i \sim \mathrm{Bin}(y_i,p)\). The next state is \(x'=y-D\); Figure~\ref{fig:mdp-small-example} illustrates this transition support for \(b=(2,1)\).

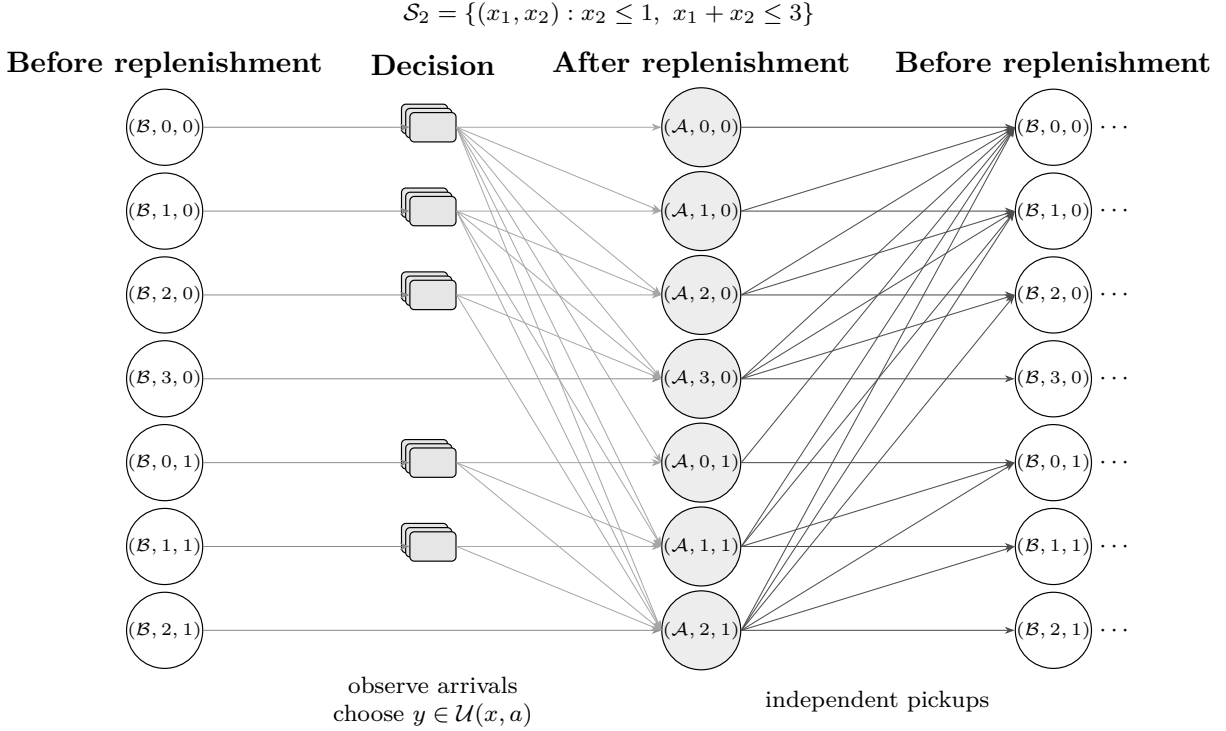
\begin{figure}[htbp]
\centering
\resizebox{\textwidth}{!}{
\begin{tikzpicture}[
    >=Latex,
    state/.style={circle,draw,fill=white,minimum size=8mm,inner sep=0pt,font=\tiny},
    poststate/.style={circle,draw,fill=black!7,minimum size=8mm,inner sep=0pt,font=\tiny},
    decision/.style={rectangle,draw,rounded corners=1.5pt,fill=black!10,align=center,minimum width=5.5mm,minimum height=3.5mm,inner sep=0.5pt,font=\scriptsize},
    stateedge/.style={-{Stealth[length=1.1mm,width=0.8mm]},draw=black!45,line width=0.25pt},
    decisionedge/.style={-{Stealth[length=1.1mm,width=0.8mm]},draw=black!35,line width=0.25pt},
    pickupedge/.style={-{Stealth[length=1.2mm,width=0.85mm]},draw=black!70,line width=0.35pt},
    note/.style={align=center,font=\scriptsize},
]

\node[font=\small\bfseries] at (0,3.75) {Before replenishment};
\node[font=\small\bfseries] at (3.2,3.75) {Decision};
\node[font=\small\bfseries] at (6.4,3.75) {After replenishment};
\node[font=\small\bfseries] at (10.6,3.75) {Before replenishment};
\node[note] at (5.3,4.35)
{$\mathcal S_2=\{(x_1,x_2):x_2\le 1,\ x_1+x_2\le 3\}$};

\node[state] (l00) at (0,3) {$(\mathcal B,0,0)$};
\node[state] (l10) at (0,2) {$(\mathcal B,1,0)$};
\node[state] (l20) at (0,1) {$(\mathcal B,2,0)$};
\node[state] (l30) at (0,0) {$(\mathcal B,3,0)$};
\node[state] (l01) at (0,-1) {$(\mathcal B,0,1)$};
\node[state] (l11) at (0,-2) {$(\mathcal B,1,1)$};
\node[state] (l21) at (0,-3) {$(\mathcal B,2,1)$};

\node[decision] (d00a) at (3.10,3.10) {};
\node[decision] (d00b) at (3.15,3.05) {};
\node[decision] (d00) at (3.20,3.00) {};
\node[decision] (d10a) at (3.10,2.10) {};
\node[decision] (d10b) at (3.15,2.05) {};
\node[decision] (d10) at (3.20,2.00) {};
\node[decision] (d20a) at (3.10,1.10) {};
\node[decision] (d20b) at (3.15,1.05) {};
\node[decision] (d20) at (3.20,1.00) {};
\node[decision] (d01a) at (3.10,-0.90) {};
\node[decision] (d01b) at (3.15,-0.95) {};
\node[decision] (d01) at (3.20,-1.00) {};
\node[decision] (d11a) at (3.10,-1.90) {};
\node[decision] (d11b) at (3.15,-1.95) {};
\node[decision] (d11) at (3.20,-2.00) {};

\node[poststate] (m00) at (6.4,3) {$(\mathcal A,0,0)$};
\node[poststate] (m10) at (6.4,2) {$(\mathcal A,1,0)$};
\node[poststate] (m20) at (6.4,1) {$(\mathcal A,2,0)$};
\node[poststate] (m30) at (6.4,0) {$(\mathcal A,3,0)$};
\node[poststate] (m01) at (6.4,-1) {$(\mathcal A,0,1)$};
\node[poststate] (m11) at (6.4,-2) {$(\mathcal A,1,1)$};
\node[poststate] (m21) at (6.4,-3) {$(\mathcal A,2,1)$};

\node[state] (r00) at (10.6,3) {$(\mathcal B,0,0)$};
\node[state] (r10) at (10.6,2) {$(\mathcal B,1,0)$};
\node[state] (r20) at (10.6,1) {$(\mathcal B,2,0)$};
\node[state] (r30) at (10.6,0) {$(\mathcal B,3,0)$};
\node[state] (r01) at (10.6,-1) {$(\mathcal B,0,1)$};
\node[state] (r11) at (10.6,-2) {$(\mathcal B,1,1)$};
\node[state] (r21) at (10.6,-3) {$(\mathcal B,2,1)$};
\node[font=\scriptsize] at (11.35,3) {$\cdots$};
\node[font=\scriptsize] at (11.35,2) {$\cdots$};
\node[font=\scriptsize] at (11.35,1) {$\cdots$};
\node[font=\scriptsize] at (11.35,0) {$\cdots$};
\node[font=\scriptsize] at (11.35,-1) {$\cdots$};
\node[font=\scriptsize] at (11.35,-2) {$\cdots$};
\node[font=\scriptsize] at (11.35,-3) {$\cdots$};

\foreach \s/\d in {l00/d00,l10/d10,l20/d20,l01/d01,l11/d11}{
    \draw[stateedge] (\s.east) -- (\d.west);
}
\draw[stateedge] (l30.east) -- (m30.west);
\draw[stateedge] (l21.east) -- (m21.west);

\foreach \m in {m00,m10,m20,m30,m01,m11,m21}{
    \draw[decisionedge] (d00.east) -- (\m.west);
}
\foreach \m in {m10,m20,m30,m11,m21}{
    \draw[decisionedge] (d10.east) -- (\m.west);
}
\foreach \m in {m20,m30,m21}{
    \draw[decisionedge] (d20.east) -- (\m.west);
}
\foreach \m in {m01,m11,m21}{
    \draw[decisionedge] (d01.east) -- (\m.west);
}
\foreach \m in {m11,m21}{
    \draw[decisionedge] (d11.east) -- (\m.west);
}

\draw[pickupedge] (m00.east) -- (r00.west);
\foreach \r in {r00,r10}{
    \draw[pickupedge] (m10.east) -- (\r.west);
}
\foreach \r in {r00,r10,r20}{
    \draw[pickupedge] (m20.east) -- (\r.west);
}
\foreach \r in {r00,r10,r20,r30}{
    \draw[pickupedge] (m30.east) -- (\r.west);
}
\foreach \r in {r00,r01}{
    \draw[pickupedge] (m01.east) -- (\r.west);
}
\foreach \r in {r00,r10,r01,r11}{
    \draw[pickupedge] (m11.east) -- (\r.west);
}
\foreach \r in {r00,r10,r20,r01,r11,r21}{
    \draw[pickupedge] (m21.east) -- (\r.west);
}

\node[note] at (3.2,-3.85) {observe arrivals\\choose $y\in\mathcal U(x,a)$};
\node[note] at (8.5,-3.85) {independent pickups};

\end{tikzpicture}
}
\figcaption{MDP transition support for \(b=(2,1)\) where \(\mathcal B\) and \(\mathcal A\) denote before- and after-replenishment copies; stacked rectangles denote arrival-contingent decisions; pickup arcs map post-decision states to feasible next before states. \label{fig:mdp-small-example}}
\end{figure}

The infinite-horizon average-cost MDP is described by the Bellman equation
\begin{equation}\label{eq:nsize-bellman}
h(x)+g
=
\mathbb E_A\left[
\min_{y\in\mathcal U(x,A)}
\left\{
\sum_{i=1}^n \pi_i(x_i+A_i-y_i)
+\mathbb E\bigl[h(y-D)\mid y\bigr]
\right\}
\right],
\qquad x\in\mathcal S_n.
\end{equation}
Here \(g\) is the optimal long-run average rejection cost per cycle and \(h\) is the relative value function. 
Under the assumptions $0<p<1$ and $\Pr(X_t=0)>0$, every state can reach the empty state with positive probability under any stationary policy, so the induced Markov chain is unichain.

The number of feasible states is
\[
|\mathcal S_n|
=
\sum_{x_n=0}^{b_n}
\sum_{x_{n-1}=0}^{b_{n-1}+b_n-x_n}
\cdots
\sum_{x_2=0}^{B_2-\sum_{k=3}^n x_k}
\left(B_1-\sum_{k=2}^n x_k+1\right).
\]
Thus the exact MDP is valid for any number of sizes, but its state space grows quickly with both the number of sizes and the locker counts. This is why exact value iteration is used mainly as a validation tool for small instances.

Consider the two-size case; that is, $n=2$. Then \(\mathcal S_2=\{(x_1,x_2)\in\mathbb Z_+^2: x_2\le b_2,\ x_1+x_2\le b_1+b_2\}\), and the final admission decision can be summarized by the target number of size-1 parcels retained after size-2 parcels have been processed. This makes the two-size case useful for broad exact sweeps and for illustrating the opportunity-cost logic. 

\section{Relative Value Iteration and the AA Policy}\label{sec:policy-analysis}

\subsection{Relative Value Iteration}\label{sec:vi}
The exact computations in this paper use relative value iteration (Algorithm~\ref{alg:nsize-vi}) for the finite-state average-cost MDP in \eqref{eq:nsize-bellman}. The implementation maintains an approximation $v^{(k)}$ of the bias function, applies the Bellman operator, and normalizes each iteration by subtracting the value of a reference state, taken to be the empty state $0=(0,\dots,0)$. The resulting scalar \(g^{(k)}=(Tv^{(k)})(0)\) is the current estimate of the optimal average daily rejection cost.

The implementation precomputes binomial pickup probabilities for all feasible occupancies. In each iteration, it first evaluates the continuation table \(H^{(k)}(y)=\mathbb E[v^{(k)}(y-D)\mid y]\) for all feasible post-decision vectors $y$. It then updates each state by minimizing \(\sum_{i=1}^n \pi_i(x_i+a_i-y_i)+H^{(k)}(y)\) over $y\in\mathcal U(x,a)$ and averaging over arrivals. The minimization is factored through tail loads \(Y_i=\sum_{k=i}^n y_k\), which avoids enumerating infeasible placements. For Poisson arrivals, the sums are evaluated exactly up to saturation thresholds and the remaining Poisson tails are aggregated analytically.

\begin{algorithm}[htbp]
\caption{Relative value iteration for the multi-size locker admission MDP}
\label{alg:nsize-vi}
\begin{algorithmic}[1]
\State \textbf{Input:} capacities $b$, pickup probability $p$, arrivals, penalties $\pi$, tolerance $\varepsilon$, maximum iterations $K_{\max}$
\State Enumerate the feasible state space $\mathcal S_n$
\State Set $v^{(0)}(x)\gets 0$ for all $x\in\mathcal S_n$ and choose reference state $0=(0,\dots,0)$
\State Precompute binomial pickup probabilities for all feasible occupancies
\For{$k=0,1,\dots,K_{\max}-1$}
    \State Compute $H^{(k)}(y)=\mathbb E[v^{(k)}(y-D)\mid y]$ for all feasible post-decision vectors $y$
    \For{each state $x\in\mathcal S_n$}
        \State Evaluate $(Tv^{(k)})(x)$ by averaging over arrivals and minimizing over $\mathcal U(x,a)$
    \EndFor
    \State Set $g^{(k)}\gets (Tv^{(k)})(0)$
    \State Normalize: $v^{(k+1)}(x)\gets (Tv^{(k)})(x)-g^{(k)}$ for all $x\in\mathcal S_n$
    \State Compute the span error \(\delta^{(k)}=\max_x(v^{(k+1)}(x)-v^{(k)}(x))-\min_x(v^{(k+1)}(x)-v^{(k)}(x))\)
    \If{$\delta^{(k)}\le \varepsilon$}
        \State \textbf{break}
    \EndIf
\EndFor
\State \textbf{Output:} average-cost estimate $g^{(k)}$, normalized value table $v^{(k+1)}$, and an optimal stationary admission policy
\end{algorithmic}
\end{algorithm}

For $n=2$, this algorithm is particularly fast because the final replenishment decision becomes one-dimensional after the size-2 occupancy has been fixed. The same implementation framework is used for the general $n$-size case, but exact value iteration is practical only for small systems.

\subsection{AA Policy}\label{sec:always-accept}
The always-accept (AA) policy processes arrivals after relocation in decreasing order of size. Size-$n$ parcels are inserted into available size-$n$ lockers, with any spillover rejected. Size-$(n-1)$ parcels are then inserted first into available size-$(n-1)$ lockers and then, if needed, into available size-$n$ lockers. More generally, a size-$i$ parcel is inserted first into size-$i$ lockers, then into size-$(i+1)$ lockers, and so on up to size $n$; only residual spillover is rejected.

The AA policy has a recursive representation once the capacity occupied by existing smaller-size parcels is reserved. Using the spillover quantities \(s_i(x)\) defined in Section~\ref{sec:nsize-model}, \(s_{i-1}(x)\) is the number of existing parcels from sizes \(1,\dots,i-1\) that must occupy lockers from sizes \(i,\dots,n\) after the morning relocation step. Starting with the largest size and moving downward, AA sets
\[
y_i^{AA}(x,a)
=
\min\left\{
x_i+a_i,\;
B_i-\sum_{k=i+1}^n y_k^{AA}(x,a)-s_{i-1}(x)
\right\},
\qquad i=n,n-1,\dots,1.
\]
The corresponding size-$i$ rejection count is \(r_i^{AA}(x,a)=x_i+a_i-y_i^{AA}(x,a)\).
This recursion preserves all parcels already present and decides only how many new arrivals to admit. The key question is whether it can be optimal to reject a size-$i$ arrival while a compatible larger locker is available.

For the sufficient-condition argument, view interface $i$ as the boundary between a smaller-size aggregate $\{1,\dots,i\}$ and a larger-size aggregate $\{i+1,\dots,n\}$. Define \(H_i=\sum_{k=i+1}^n b_k\), \(i=1,\dots,n-1\). Suppose that immediately after replenishment, exactly $m\in\{0,\dots,H_i\}$ parcels from sizes $1,\dots,i$ occupy lockers from sizes $i+1,\dots,n$. Then the total number of smaller-size parcels in the system is $C_i+m$, because all lockers of sizes $1,\dots,i$ are full whenever $m>0$. Let \(D_{i,m}\sim \mathrm{Bin}(C_i+m,p)\) be the number of those smaller-size parcels that depart during the day. After the next morning relocation step, the number of smaller-size parcels that still occupy lockers from sizes $i+1,\dots,n$ is \(R_{i,m}=(m-D_{i,m})^+\). Hence the event that current smaller-size use of larger-size capacity can still matter one cycle later is \(\{R_{i,m}>0\}=\{D_{i,m}<m\}\). Define \(q_{i,m}:=\Pr(D_{i,m}<m)=\Pr(\mathrm{Bin}(C_i+m,p)<m)\).
A coupling argument gives that $q_{i,m}$ is non-decreasing in $m$, so
\[
q_{i,m}\le q_i^{\max}:=q_{i,H_i}
=
\Pr\!\left(\mathrm{Bin}\!\left(\sum_{k=1}^n b_k,p\right)<H_i\right).
\]

Thus \(q_i^{\max}\) upper-bounds the probability that spillover across the boundary above size \(i\) survives one cycle and still occupies larger-size capacity after the next relocation step. The current spillover episode matters only until it first clears; later spillover is caused by later AA admissions and is charged separately.

\subsection{AA Optimality}\label{sec:always-accept-opt}
\begin{theorem}\label{thm:nsize-always-accept}
The AA policy is optimal when \(\pi_i \ge \pi_n q_i^{\max}\) for every size \(i=1,\dots,n-1\).
\end{theorem}

\begin{proof}
We use a one-step deviation argument for the finite unichain average-cost MDP, following the standard average-cost policy-improvement logic for finite Markov decision processes \citep[Chapter~8]{Puterman1994}. It is enough to show that, after any state and arrival realization, no marginal rejection of a parcel that AA would admit can improve the Bellman comparison in \eqref{eq:nsize-bellman}. The feasible deviations from \(y^{AA}(x,a)\) form
the set of admissible post-decision vectors \(y\in\mathcal U(x,a)\) with
\(y\le y^{AA}(x,a)\). Any such \(y\) can be reached from \(y^{AA}(x, a)\) by a
finite sequence of single-size unit decrements, processed in the same
decreasing size order as AA. If no single decrement strictly improves
the Bellman comparison, then no finite sequence of such decrements can improve
it either. Hence, it suffices to verify the marginal case.

Fix a size \(i<n\). We call a boundary-\(i\) spillover episode a consecutive
sequence of replenishment cycles in which at least one parcel from sizes
\(1,\dots,i\) is stored in a locker from sizes \(i+1,\dots,n\). The episode
starts when AA inserts a size-\(i\) parcel into a larger locker, and it ends
immediately after the first relocation step, at which no parcel from sizes
\(1,\dots,i\) remains in the larger lockers. The marginal size-\(i\) admission
that starts the episode gives a deterministic saving of \(\pi_i\),
because rejecting it would have incurred that penalty.

The only possible downside of this marginal admission is that, before the
episode ends, it may occupy capacity that could otherwise have admitted a
larger-size parcel. This downside is limited in two ways. First, if the
spillover episode ends at the next relocation step, then the marginal admission
has no further effect on larger-size parcels, either in the next cycle or in
any later cycle. When the episode starts with \(m\) lower-side parcels in
higher-side lockers, the probability that it does not end at the next
relocation step is \(q_{i,m}\le q_i^{\max}\). Since \(p>0\), the same bound also implies that a spillover episode clears almost surely after repeated cycles, so the sample-path comparison below charges only a finite episode.

Second, conditional on the episode surviving, the marginal admission can be
charged for at most one future rejection. If the rejected parcel is from one of
the sizes \(1,\dots,i\), its penalty is at most \(\pi_i\), and the saving from admitting the marginal size-\(i\) parcel already pays for that
charge. The only residual risk is therefore that the marginal admission causes a future rejection of a larger-size parcel.

To bound this residual risk, consider the first future larger-size parcel, if
any, that is rejected because the marginal smaller-size parcel was previously
admitted. In the comparison system in which the marginal parcel was rejected,
this larger-size parcel would be admitted. From that point onward, use a sample-path coupling: assign the accepted larger-size parcel in the comparison system and the smaller-size spillover token in the AA system the same memoryless pickup realization. Until their common pickup time, the two systems differ by the identity of the parcel occupying one larger-size locker, not by the amount of larger-size capacity occupied. After that pickup time, both systems release that capacity. Thus this token cannot be responsible for a second larger-size rejection. If a later smaller-size parcel enters a larger locker after the original episode has cleared, that is a new AA marginal admission and is charged separately.
Therefore the expected residual opportunity cost of a marginal size-\(i\)
spillover admission is at most \(\pi_n q_i^{\max}\): with probability at most
\(q_i^{\max}\) it may be charged one rejected larger-size parcel, and the
penalty of such a parcel is at most \(\pi_n\). Under the stated condition
\(\pi_i\ge \pi_n q_i^{\max}\), the saving from accepting the marginal
size-\(i\) parcel weakly dominates this worst-case residual opportunity cost.
Hence no marginal spillover rejection is improving. Marginal rejections that do
not use larger lockers are even simpler: by the same memoryless coupling, accepting such a parcel can at worst displace one future parcel of the same size, so it cannot be worse than rejecting it immediately. Thus, no one-step deviation from AA is improving; that is, AA is optimal.
\end{proof}

\begin{remark}
When $n=2$, the condition in Theorem~\ref{thm:nsize-always-accept} reduces to
\(\pi_1\ge \pi_2\Pr(\mathrm{Bin}(b_1+b_2,p)<b_2)\) as a small parcel placed in a large locker can matter in the future only if fewer than the current number of small spillover parcels are picked up before the next relocation step.
\end{remark}

\begin{remark} 
Theorem~\ref{thm:nsize-always-accept} replaces the exact continuation-value comparison by a one-cycle worst-case bound at each interface, and upper-bounds the loss from one blocked future parcel by the maximal penalty $\pi_n$. 
\end{remark}

In many realistic systems, $H_i$ is small relative to the total capacity and $p$ is not too small, so the persistence probabilities \(q_i^{\max}\) can be quite small.
The numerical experiments use exact value iteration to evaluate cases not covered by the conservative sufficient condition.

When Theorem~\ref{thm:nsize-always-accept} certifies AA, evaluation reduces to a fixed-policy problem where the admission decision is fixed, and the system evolves as a finite discrete-time Markov chain on \(\mathcal S_n\) rather than as a controlled MDP. The average rejection cost can be computed by constructing the transition matrix induced by the AA policy, solving for its stationary distribution, and averaging the one-period rejection cost over that distribution. This fixed-policy DTMC calculation is much cheaper than value iteration because it avoids Bellman minimization over feasible post-decision states. Moreover, under the theorem's condition, the DTMC value is not an approximation to the MDP value; it is the exact optimal average rejection cost.

Figure~\ref{fig:aa-dtmc-small-example} illustrates this fixed-policy DTMC for a small two-size system with $b=(2,1)$. The graph is drawn as a two-stage directed support graph. The left and right columns are duplicate copies of the same before-replenishment DTMC state space, marked by $\mathcal B$; they are separated only to make one transition cycle visually clear. A before-replenishment state first moves to an after-replenishment state, marked by $\mathcal A$, determined by random arrivals and the AA policy. Then pickups move this state to a state in the next before-replenishment copy, and the ellipses indicate that the same cycle repeats.

\begin{figure}[htbp]
\centering
\resizebox{\textwidth}{!}{
\begin{tikzpicture}[
    >=Latex,
    state/.style={circle,draw,fill=white,minimum size=8mm,inner sep=0pt,font=\tiny},
    poststate/.style={circle,draw,fill=black!7,minimum size=8mm,inner sep=0pt,font=\tiny},
    admitedge/.style={-{Stealth[length=1.1mm,width=0.8mm]},draw=black!35,line width=0.25pt},
    pickupedge/.style={-{Stealth[length=1.2mm,width=0.85mm]},draw=black!70,line width=0.35pt},
    note/.style={align=center,font=\scriptsize},
]

\node[font=\small\bfseries] at (0,3.75) {Before replenishment};
\node[font=\small\bfseries] at (5.3,3.75) {After replenishment};
\node[font=\small\bfseries] at (10.6,3.75) {Before replenishment};
\node[note] at (5.3,4.35)
{$\mathcal S_2=\{ (x_1,x_2):x_2\le 1,\ x_1+x_2\le 3\}$};

\node[state] (l00) at (0,3) {$(\mathcal B,0,0)$};
\node[state] (l10) at (0,2) {$(\mathcal B,1,0)$};
\node[state] (l20) at (0,1) {$(\mathcal B,2,0)$};
\node[state] (l30) at (0,0) {$(\mathcal B,3,0)$};
\node[state] (l01) at (0,-1) {$(\mathcal B,0,1)$};
\node[state] (l11) at (0,-2) {$(\mathcal B,1,1)$};
\node[state] (l21) at (0,-3) {$(\mathcal B,2,1)$};

\node[poststate] (m00) at (5.3,3) {$(\mathcal A,0,0)$};
\node[poststate] (m10) at (5.3,2) {$(\mathcal A,1,0)$};
\node[poststate] (m20) at (5.3,1) {$(\mathcal A,2,0)$};
\node[poststate] (m30) at (5.3,0) {$(\mathcal A,3,0)$};
\node[poststate] (m01) at (5.3,-1) {$(\mathcal A,0,1)$};
\node[poststate] (m11) at (5.3,-2) {$(\mathcal A,1,1)$};
\node[poststate] (m21) at (5.3,-3) {$(\mathcal A,2,1)$};

\node[state] (r00) at (10.6,3) {$(\mathcal B,0,0)$};
\node[state] (r10) at (10.6,2) {$(\mathcal B,1,0)$};
\node[state] (r20) at (10.6,1) {$(\mathcal B,2,0)$};
\node[state] (r30) at (10.6,0) {$(\mathcal B,3,0)$};
\node[state] (r01) at (10.6,-1) {$(\mathcal B,0,1)$};
\node[state] (r11) at (10.6,-2) {$(\mathcal B,1,1)$};
\node[state] (r21) at (10.6,-3) {$(\mathcal B,2,1)$};
\node[font=\scriptsize] at (11.35,3) {$\cdots$};
\node[font=\scriptsize] at (11.35,2) {$\cdots$};
\node[font=\scriptsize] at (11.35,1) {$\cdots$};
\node[font=\scriptsize] at (11.35,0) {$\cdots$};
\node[font=\scriptsize] at (11.35,-1) {$\cdots$};
\node[font=\scriptsize] at (11.35,-2) {$\cdots$};
\node[font=\scriptsize] at (11.35,-3) {$\cdots$};

\foreach \m in {m00,m10,m20,m30,m01,m11,m21}{
    \draw[admitedge] (l00.east) -- (\m.west);
}
\foreach \m in {m10,m20,m30,m11,m21}{
    \draw[admitedge] (l10.east) -- (\m.west);
}
\foreach \m in {m20,m30,m21}{
    \draw[admitedge] (l20.east) -- (\m.west);
}
\draw[admitedge] (l30.east) -- (m30.west);
\foreach \m in {m01,m11,m21}{
    \draw[admitedge] (l01.east) -- (\m.west);
}
\foreach \m in {m11,m21}{
    \draw[admitedge] (l11.east) -- (\m.west);
}
\draw[admitedge] (l21.east) -- (m21.west);

\draw[pickupedge] (m00.east) -- (r00.west);
\foreach \r in {r00,r10}{
    \draw[pickupedge] (m10.east) -- (\r.west);
}
\foreach \r in {r00,r10,r20}{
    \draw[pickupedge] (m20.east) -- (\r.west);
}
\foreach \r in {r00,r10,r20,r30}{
    \draw[pickupedge] (m30.east) -- (\r.west);
}
\foreach \r in {r00,r01}{
    \draw[pickupedge] (m01.east) -- (\r.west);
}
\foreach \r in {r00,r10,r01,r11}{
    \draw[pickupedge] (m11.east) -- (\r.west);
}
\foreach \r in {r00,r10,r20,r01,r11,r21}{
    \draw[pickupedge] (m21.east) -- (\r.west);
}
\node[note] at (2.65,-3.85) {random arrivals\\and AA admission};
\node[note] at (7.95,-3.85) {independent pickups};

\end{tikzpicture}
}
\figcaption{DTMC transition support induced by AA for \(b=(2,1)\). The before-replenishment state space is duplicated to show one cycle: arrivals and AA determine an after-replenishment state, and pickups map it to component-wise smaller before states in the next copy. \label{fig:aa-dtmc-small-example}}

\end{figure}
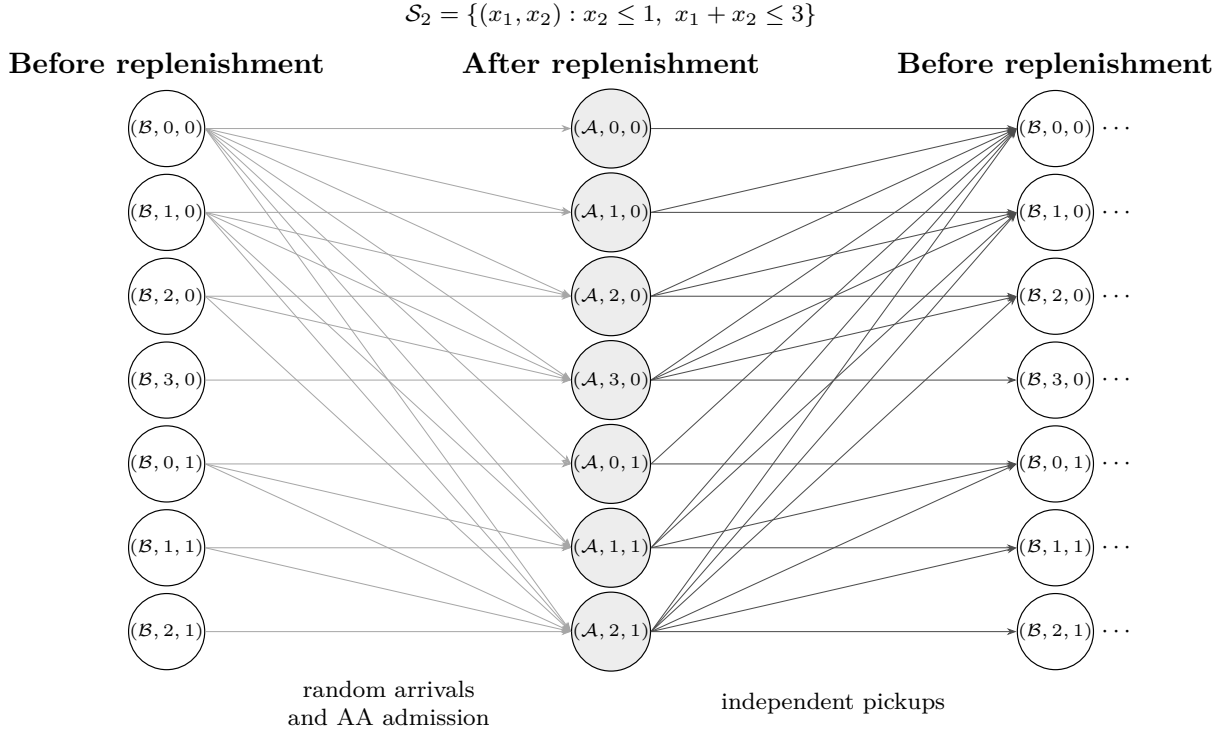

\section{Numerical Evidence on the Admission Policy Problem}\label{sec:numerical-evidence}

The numerical analysis tests whether AA is optimal across broad parameter ranges. Robustness checks for non-memoryless pickup times are reported in the supplementary material. We begin with two-size systems, which permit broad exact sweeps. 
To report demand intensity on a common scale across sizes, define \(\rho_i=\lambda_i/(p b_i)\). 
Under the no-spillover benchmark, a parcel stays in the system for an average of $1/p$ replenishment cycles, so $\lambda_i/p$ is the corresponding mean number of size-$i$ parcels in the system. Dividing by $b_i$ gives the mean load per locker size.

\begin{table}[htbp]
\centering
\caption{Two-size experiments}
\label{tab:load-sweep-summary}
\begin{tabular}{p{0.34\textwidth}p{0.56\textwidth}}
\toprule
Capacities & \((b_1,b_2)\in\{(20,10),(30,10),(40,20),\)\\
& \((60,20),(60,30),(90,30)\}\) \\
Pickup probabilities & $p\in\{0.25,0.33,0.5\}$ \\
Load grid & \(\rho_1,\rho_2\in\{0.8,1.0,1.2\}\), implemented by \(\lambda_i=\rho_i p b_i\), \(i=1,2\) \\
Penalty grid & $\pi_1=1$, $\pi_2\in\{2,4,8,16\}$ \\
Number of instances & $6\times 3\times 3\times 3\times 4=648$ \\
Iterations to convergence & min $=17$, median $=41$, max $=72$ \\
Solve time per instance (seconds) & min $=0.0014$, mean $=0.096$, max $=0.481$ \\
State-space size & from $286$ to $3{,}286$ states \\
\bottomrule
\end{tabular}
\end{table}

Table~\ref{tab:load-sweep-summary} summarizes an experiment varying capacities, pickup probability, normalized loads, and the penalty ratio. Computationally, the exact two-size MDP remains tractable, all 648 instances converged, with state spaces ranging from 286 to 3,286 states and solve times below 0.5 seconds.

In 647 of the 648 exact-value-iteration runs, AA is the optimal admission policy; that is, after the size-2 parcels have been admitted, every feasible size-1 parcel is accepted into the remaining compatible locker capacity. The only exception occurs in the smallest system, \(b=(b_1,b_2)=(20,10)\), under the combination \(p=0.25\), \((\rho_1,\rho_2)=(1.2,1.2)\), and \(\pi_2/\pi_1=16\). A tighter value-iteration run keeps this instance classified as non-AA, so we do not treat it as a stopping-tolerance artifact. However, an independent fixed-policy evaluation of AA gives a relative gap of only \(4.7\times 10^{-8}\). Thus, even in this exceptional case, AA is practically indistinguishable from optimal.

Figures~\ref{fig:always-accept-certification1} and~\ref{fig:always-accept-certification} show where AA is certified as optimal over a wider capacity grid with  $\rho_1=\rho_2=1$. Diagonally striped cells correspond to cases in which Theorem~\ref{thm:nsize-always-accept} applies. Dotted cells correspond to cases in which the exact two-size value iteration implies optimality for AA. White cells with numeric labels are cases in which AA is not optimal; the number shown is the relative gap $(g^{AA}-g^*)/g^*$. Figure~\ref{fig:always-accept-certification} deliberately uses the stress-test ratio \(\pi_2/\pi_1=20\), beyond the main two-size grid, to show where the AA policy begins to degrade in the smallest and slowest-pickup systems.

\begin{figure}[htbp]
\centering
\includegraphics[width=\textwidth]{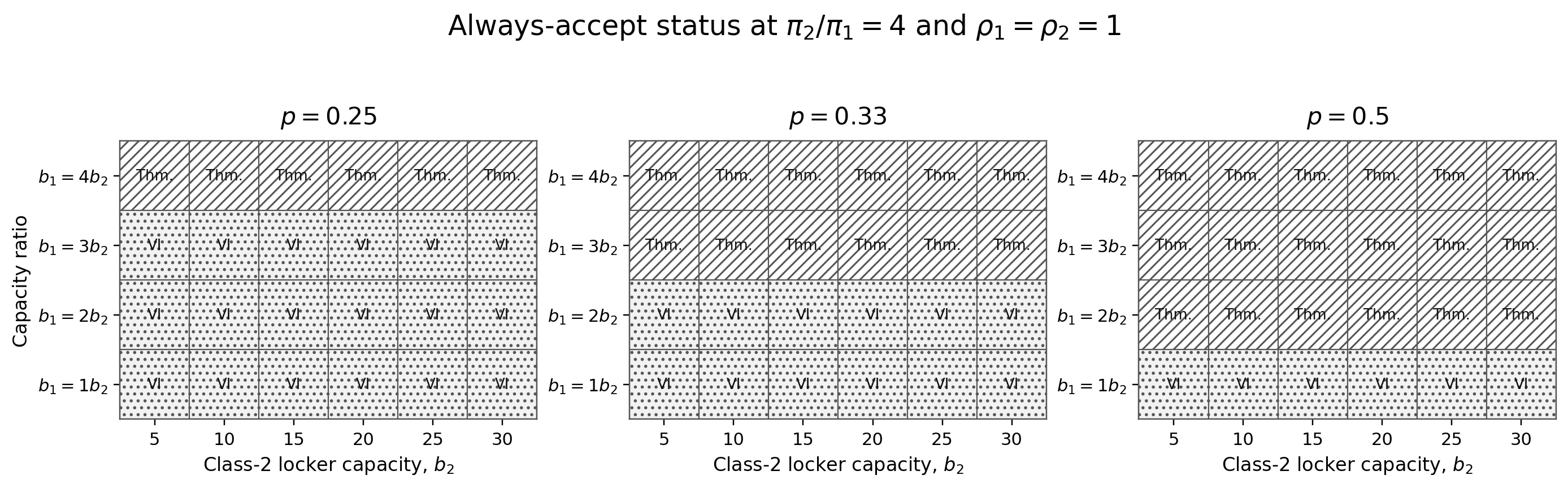}
\figcaption{AA optimality for \(\pi_2/\pi_1=4\) and \(\rho_1=\rho_2=1\). Optimality is obtained by Theorem~\ref{thm:nsize-always-accept} in striped cells and by exact value iteration in dotted cells. \label{fig:always-accept-certification1}}
\end{figure}

\begin{figure}[htbp]
\centering
\includegraphics[width=\textwidth]{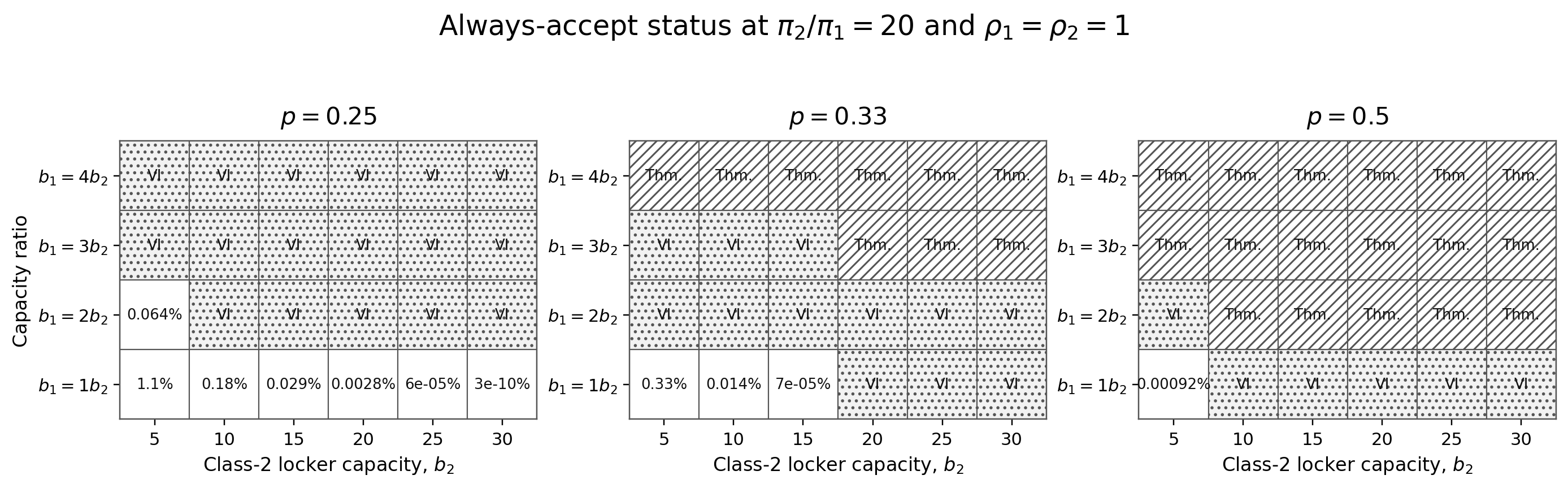}
\figcaption{Stress-test AA optimality for \(\pi_2/\pi_1=20\) and \(\rho_1=\rho_2=1\). White cells report the relative gap. \label{fig:always-accept-certification}}
\end{figure}

We next examine instances with  $n=3$ and $4$ parcel classes. Table~\ref{tab:nsize-experiment-design} summarizes the corresponding parameter grids.
Systems with three sizes, small, medium, and large, are used in empirical studies of locker configuration \citep{RanjbariEtAl2023RightSize}, while locker-bank design models use module data with small, medium, large, and extra-large sizes when optimizing physical layouts \citep{FaugereMontreuil2020,Kahr2022}. Thus, the tested cases cover the common small--medium--large setting and an extension with an extra-large size, while the MDP formulation itself allows any number of sizes. In all capacity vectors, the number of lockers decreases as size increases, reflecting the fact that larger lockers are usually less numerous. 

\begin{table}[htbp]
\centering
\caption{Multi-size numerical experiment design}
\label{tab:nsize-experiment-design}
\small
\begin{tabular}{@{}>{\raggedright\arraybackslash}p{0.25\textwidth}>{\raggedright\arraybackslash}p{0.68\textwidth}@{}}
\midrule
Capacity vectors & 3 sizes: $(12,4,2)$, $(16,5,3)$, $(24,8,4)$. 4 sizes: $(12,3,2,1)$, $(16,4,2,1)$, $(20,5,3,1)$. \\
Penalty vectors & 3 sizes: $(1,2,4)$, $(1,2,8)$, $(1,4,16)$. 4 sizes: $(1,2,4,8)$, $(1,2,4,16)$, $(1,3,9,27)$. \\
Pickup probabilities & $p\in\{0.25,0.33,0.5\}$. \\
Balanced load profiles & $(0.8,\dots,0.8)$, $(1,\dots,1)$, and $(1.2,\dots,1.2)$. \\
Unbalanced load profiles & 3 sizes: $(1.2,1,0.8)$, $(0.8,1,1.2)$. 4 sizes: $(1.2,1.05,0.9,0.75)$, $(0.75,0.9,1.05,1.2)$. \\
Total instances & $270$. \\
\bottomrule
\end{tabular}
\normalsize
\vspace{-0.5\baselineskip}
\end{table}

All 270 instances converged under the C++ value iteration implementation. The state-space sizes range from $280$ to $1{,}715$ states, the iteration counts range from $19$ to $67$, and the mean solve time is $9.5$ seconds, with a median of $2.2$ seconds and a maximum of $44.7$ seconds. Thus, these examples remain small enough for exact optimization, but they already illustrate the rapid dimensional growth discussed in Section~\ref{sec:nsize-model}.

Table~\ref{tab:nsize-sweep-summary} reports the main outcomes. The one-cycle condition in Theorem~\ref{thm:nsize-always-accept} certifies $85$ of the $270$ instances. Exact value iteration finds that AA is optimal in $221$ instances. In the remaining $49$ instances, exact value iteration finds a reservation policy with a lower average cost, but the economic effect is very small: the largest relative gap \((g^{AA}-g^*)/g^*\) is $0.1004\%$, the mean positive gap is $0.0171\%$, and the median positive gap is $0.0052\%$.

\begin{table}[htbp]
\centering
\caption{Exact multi-size value iteration sweep}
\label{tab:nsize-sweep-summary}
\small
\begin{tabular}{lcccccc}
\toprule
$n$ & Inst. & States & AA optimal & Thm.~\ref{thm:nsize-always-accept} & Max gap \\
\midrule
3 & 135 & $280$--$1{,}655$ & 117/135 & 45 & $0.0688\%$ \\
4 & 135 & $564$--$1{,}715$ & 104/135 & 40 & $0.1004\%$ \\
\bottomrule
\end{tabular}
\normalsize
\vspace{-0.5\baselineskip}
\end{table}

The pickup probability is the clearest driver of the results. Table~\ref{tab:nsize-sweep-by-p} shows that AA is optimal for all $90$ instances with $p=0.5$, with $85$ of them being certified by Theorem~\ref{thm:nsize-always-accept}. At $p=0.33$, AA is optimal in $83$ of $90$ instances, and the largest gap in the remaining cases is only $0.0043\%$. The nontrivial deviations are concentrated almost entirely at $p=0.25$, where parcels remain in the lockers longer, and the opportunity cost of smaller-size spillover is highest.

\begin{table}[htbp]
\centering
\caption{Multi-size sweep by pickup probability}
\label{tab:nsize-sweep-by-p}
\small
\begin{tabular}{ccccc}
\toprule
$p$ & AA Optimal & Thm.~\ref{thm:nsize-always-accept} & Max gap \\
\midrule
0.25 & 48/90 & 0 &  $0.1004\%$ \\
0.33 & 83/90 & 0 &  $0.0043\%$ \\
0.50 & 90/90 & 85 & $0.0000\%$ \\
\bottomrule
\end{tabular}
\normalsize
\end{table}

Equivalently, the instances in which AA is not optimal are the instances with long mean holding times. Since the geometric pickup model has mean holding time $1/p$, the suboptimal cases occur only at $p=0.25$ and $p=0.33$, corresponding to mean holding times of $4.0$ and approximately $3.0$ replenishment cycles. No suboptimal case appears at $p=0.5$, where the mean holding time is $2.0$ cycles. This pattern is consistent with the intuition behind the sufficient condition: a smaller-size parcel admitted into larger-size capacity can create a meaningful opportunity cost only if it is likely to remain in the system long enough to compete with future larger-size arrivals.

We ran a separate scaled three-size experiment to test whether deviations from AA become less important as the system size grows. Starting from the base composition $(6,2,1)$, we considered the scaled capacity vectors \((6\sigma,2\sigma,\sigma)\), \(\sigma=1,2\dots,5\), using the same pickup probabilities, penalty vectors, and five normalized-load profiles as in Table~\ref{tab:nsize-experiment-design}. This yields $225$ additional exact value iteration instances. Table~\ref{tab:nsize-scale-summary} summarizes the results by scale.

\begin{table}[htbp]
\centering
\caption{Scaled three-size experiment}
\label{tab:nsize-scale-summary}
\small
\begin{tabular}{ccccc}
\toprule
$\sigma$ & Capacities & AA optimal & Max gap \\
\midrule
1 & $(6,2,1)$ & 40/45 & $0.1663\%$ \\
2 & $(12,4,2)$ & 39/45 &  $0.0649\%$ \\
3 & $(18,6,3)$ & 41/45 &  $0.0264\%$ \\
4 & $(24,8,4)$ & 42/45 & $0.0045\%$ \\
5 & $(30,10,5)$ & 43/45 & $0.0014\%$ \\
\bottomrule
\end{tabular}
\normalsize
\end{table}

The maximum relative gap falls sharply from $0.1663\%$ at $\sigma=1$ to $0.0014\%$ at $\sigma=5$, and the mean gap falls from $0.0079\%$ to $0.00003\%$. In a comparison of the 45 parameter combinations across all five scales, 43 have non-increasing optimality gaps. The two exceptions are cases in which AA is optimal at $\sigma=1$, has a small positive gap at $\sigma=2$, and is optimal for larger scales. Across the tested grids, AA is effectively optimal for operational purposes. The only meaningful deviations occur in slow-pickup, high-penalty-ratio cases, and even there the relative cost gaps remain small. Thus, AA provides a strong baseline for the locker-assortment analysis. The exact MDP is useful for detecting the small reservation effects that can arise in very small systems when pickup probabilities are low and penalty ratios are steep.

\section{Locker Assortment Problem}\label{sec:locker-assortment}

We consider the locker-assortment problem: given a fixed installation cost \(f_i\) for each size-\(i\) locker, how many lockers of each size should be installed to minimize the combined facility and rejection costs? We express the facility cost as a daily equivalent cost, so it can be added to the steady-state daily rejection cost. Thus, the locker-assortment combined cost is \(g^*(b)+\sum_{i=1}^n f_i b_i\), where \(g^*(b)\) is the optimal steady-state rejection cost over all admissible policies. Since the facility cost component is linear, we focus on the steady-state rejection cost \(g^*(b)\).
Let $g^{AA}(b)$ denote the cost under the AA policy. We next examine when AA is optimal; that is, when \(g^*(b)=g^{AA}(b)\).

\subsection{Daily Rejection Cost}\label{sec:assortment-objective}

We first illustrate the daily rejection-cost function using the two-size locker-assortment sweep already summarized for the AA policy. In that experiment,  \((\lambda_1,\lambda_2,p,\pi_1,\pi_2)=(20,10,0.5,1,4)\), and the locker capacities vary over the expanded grid \(b_1=32,\dots,60\) and \(b_2=16,\dots,30\). Figure~\ref{fig:assortment-surface} plots the resulting optimal average daily penalty surface \(g^*(b_1,b_2)\). The surface decreases from the value at \((b_1,b_2)=(32,16)\) to the value at \((b_1,b_2)=(60,30)\), and all 435 points on this enlarged assortment grid still satisfy the AA condition. Thus, on this grid, \(g^*(b_1,b_2)=g^{AA}(b_1,b_2)\): the admission policy remains optimal while the cost surface can be used directly for assortment selection.

The locker-assortment surface exhibits a diminishing-returns pattern. On the tested grid, it is coordinate-wise discretely convex: every second forward difference in the large-locker direction is strictly positive, and every second forward difference in the small-locker direction is strictly positive.
The surface is supermodular on this grid. The mixed second difference
\(g^*(b_1+1,b_2+1)-g^*(b_1+1,b_2)-g^*(b_1,b_2+1)+g^*(b_1,b_2) \)
is positive throughout the grid. Because \(g^*\) is a cost function, this positive mixed difference means that the two capacity types act as partial substitutes: adding one type of locker reduces the marginal cost reduction from adding the other type. Thus, on this two-size grid, the surface exhibits smooth diminishing marginal benefits from each locker size and a substitution effect across locker sizes. It should not, however, be interpreted as evidence of \(L^\natural\)-convexity; that stronger property requires the full discrete midpoint inequality.

\begin{figure}[htbp]
\centering
\includegraphics[width=0.82\textwidth]{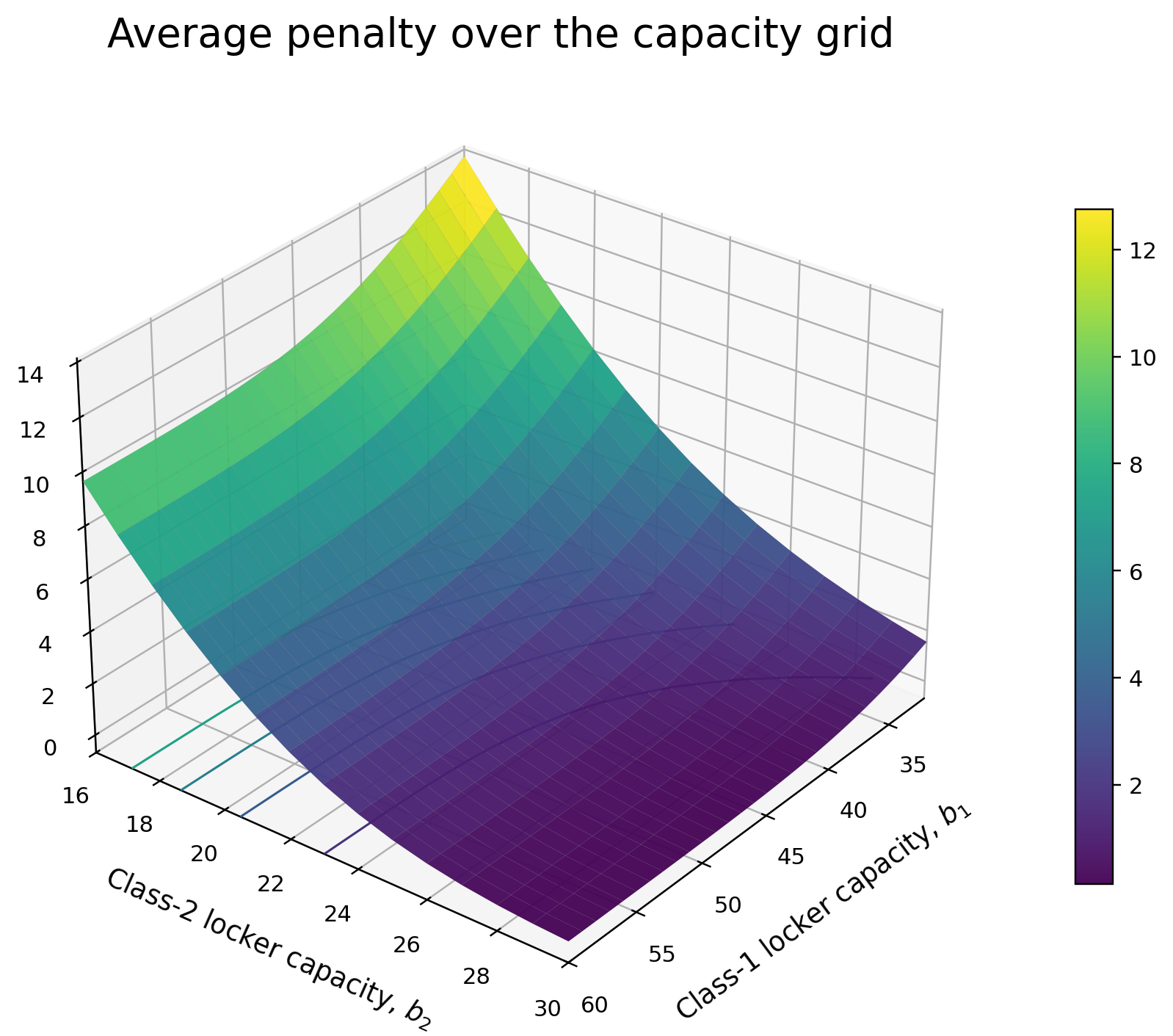}
\figcaption{Average daily penalty over the locker-assortment grid with \((\lambda_1,\lambda_2,p,\pi_1,\pi_2)=(20,10,0.5,1,4)\). \label{fig:assortment-surface}}

\end{figure}

\subsection{Discrete-Convexity}\label{sec:assortment-convexity}
Although Figure~\ref{fig:assortment-surface} suggests a regular cost surface on the tested grid, the following instances show $g^*(b)$ does not satisfy global discrete-convexity conditions such as \(L^\natural\)-convexity or \(M^\natural\)-convexity. 

\begin{counterexample}
 \label{ex:discrete-convexity-counterexample}
The optimized rejection cost function \(g^*(b)\) need not satisfy discrete-convexity properties that would make local search automatically certifiable. Because the facility term is linear, the same violations apply to the full locker-assortment objective \(C^*(b)\) for any fixed vector \(f\). First, consider the instance \(p=0.25\), \(\lambda=(0.2,0.2)\), and \(\pi=(1,2)\), where exact value iteration gives
\[
\begin{array}{c|cccc}
b & (1,1) & (1,2) & (2,1) & (2,2)\\
\hline
g^*(b)
& 0.228
& 0.088
& 0.183
& 0.060.
\end{array}
\]
AA is optimal at these four capacity vectors and the values are obtained from exact fixed-policy DTMC linear-system evaluations. Nevertheless, \(g^*(1,2)+g^*(2,1)=0.271<0.288=g^*(1,1)+g^*(2,2)\).
This violates the discrete midpoint inequality
\[
g(x)+g(y)\ge g\!\left(\left\lfloor \frac{x+y}{2}\right\rfloor\right)
          +g\!\left(\left\lceil \frac{x+y}{2}\right\rceil\right),
\]
with \(x=(1,2)\) and \(y=(2,1)\). Hence \(g^*\) is not \(L^\natural\)-convex.

Second, the optimized cost need not be supermodular. Consider instead \(p=0.25\), \(\lambda=(0.5,1)\), and \(\pi=(1,2)\).
Again, exact value iteration shows that AA is optimal at the following four capacity vectors, so the displayed values are exact fixed-policy DTMC evaluations:
\[
\begin{array}{c|cccc}
b & (1,1) & (1,2) & (2,1) & (2,2)\\
\hline
g^*(b)
& 1.8790
& 1.4736
& 1.7384
& 1.3326.
\end{array}
\]
These values satisfy \(g^*(1,1)+g^*(2,2)=3.2116<3.2120=g^*(2,1)+g^*(1,2)\),
which is the opposite of the supermodularity inequality. The same four-point inequality rules out \(M^\natural\)-convexity: take \(x=(2,2)\), \(y=(1,1)\), and move one unit of size-1 capacity from \(x\) to \(y\). Since \(x-y\) has no negative component, the \(M^\natural\)-exchange condition would require \(g^*(2,2)+g^*(1,1)\ge g^*(1,2)+g^*(2,1)\), which is false. The same inequality rules out multimodularity. For two variables, using the multimodular directions \(-e_1\), \(e_1-e_2\), and \(e_2\), take \(z=(2,1)\), \(d=-e_1\), and \(d'=e_2\). Multimodularity would require \(g^*(z+d)+g^*(z+d')\ge g^*(z)+g^*(z+d+d')\), that is, \(g^*(1,1)+g^*(2,2)\ge g^*(2,1)+g^*(1,2)\), again contradicted by the values above.
\end{counterexample}

These counterexamples do not imply that the locker-assortment objective lacks useful structure. 
They rule out only several global discrete-convexity routes that would justify a greedy or local-search algorithm. We have not proved or disproved all weaker
notions of discrete convexity or neighborhood unimodality that may hold on
practically relevant parameter ranges. Indeed, the numerical experiments are consistent with a well-behaved local-search landscape. However, the
bound-and-enumerate optimization method does not rely on convexity properties: local search is used to obtain a strong incumbent, while global optimality follows from partial enumeration of the incumbent simplex defined in Theorem~\ref{thm:incumbent-design-simplex}, valid relaxation lower bounds, and exact evaluation of the remaining candidates using the value iteration algorithm.

\subsection{Bound-and-Enumerate}\label{sec:assortment-optimization}
\begin{theorem}
\label{thm:incumbent-design-simplex}
Consider the locker assortment problem \(C^*(b)=g^*(b)+\sum_{i=1}^n f_i b_i\), where \(b\in\mathbb Z_+^n\), \(g^*(b)\ge 0\), and \(f_i>0\) is the unit cost of a size-\(i\) locker. Let \(U\) be any upper bound on the optimal objective value. Then every optimal solution \(b^*\) satisfies \(\sum_{i=1}^n f_i b_i^* \le U\).
Equivalently, at least one optimal solution is contained in the finite weighted simplex
\[
\mathcal D(U)
=
\left\{
b\in\mathbb Z_+^n:
\sum_{i=1}^n f_i b_i \le U
\right\}.
\]
More precisely, \(\mathcal D(U)\) is the set of integer lattice points in the
weighted simplex whose continuous relaxation is
\(\{b\in\mathbb R_+^n:\sum_i f_i b_i\le U\}\).
\end{theorem}

\begin{proof}
Let \(b^*\) be an optimal solution. Since \(U\) is an upper bound on the optimal objective value,
\(C^*(b^*)\le U\). Moreover, \(C^*(b^*)=g^*(b^*)+\sum_{i=1}^n f_i b_i^*\ge\sum_{i=1}^n f_i b_i^*\), because \(g^*(b^*)\ge 0\). Therefore, \(\sum_{i=1}^n f_i b_i^*\le C^*(b^*)\le U\).
Thus \(b^*\in\mathcal D(U)\).
\end{proof}

We next construct a lower bound on the cost of a given capacity vector \(b\) by relaxing the size-compatibility constraints within each tail set of sizes. Such a lower bound is instrumental when searching for an optimal solution in the simplex described by Theorem~\ref{thm:incumbent-design-simplex}.

The quantity \(r(\Lambda,B)\) is computed from the Markov chain
\(Y_t\in\{0,\ldots,B\}\), where \(Y_t\) is the number of occupied lockers
immediately after replenishment. Given \(Y_t=y\), let
\(D_t\sim\mathrm{Binomial}(y,p)\) be the number of pickups and
\(A_t\sim\mathrm{Poisson}(\Lambda)\) be the number of arrivals. Then \(Y_{t+1}=\min\{B,y-D_t+A_t\}\), and the one-period rejection count is \((y-D_t+A_t-B)^+\).
Thus \(r(\Lambda,B)\) is the stationary expectation of this rejection count.

\begin{theorem}
\label{thm:tail-pooled-lower-bound}
Consider an \(n\)-size locker system with capacities
\(b=(b_1,\ldots,b_n)\), independent Poisson arrivals with rates
\(\lambda_1,\ldots,\lambda_n\), pickup probability \(p\), and rejection
penalties
\(0<\pi_1<\pi_2<\cdots<\pi_n\).
Let \(\mu\) be any admissible stationary admission policy, and let
\(J^\mu(b)\) denote its steady-state rejection cost. For each tail set of
sizes \(i,\ldots,n\), define \(B_i(b)=\sum_{k=i}^n b_k\) and \(\Lambda_i=\sum_{k=i}^n \lambda_k\).
Let \(r(\Lambda,B)\) be the steady-state rejection rate in an idealized
single-size loss system with arrival rate \(\Lambda\), pickup probability
\(p\), and \(B\) fully pooled lockers. Then
\[
J^\mu(b)
\ge
\underline J(b)
:=
\sum_{i=1}^n
(\pi_i-\pi_{i-1})\,r(\Lambda_i,B_i(b)),
\qquad \pi_0:=0.
\]
Consequently, the same lower bound applies both to the AA policy
and to the optimal MDP value \(g^*(b)=\inf_\mu J^\mu(b)\).
\end{theorem}

\begin{proof}
Fix an admissible stationary policy \(\mu\), and let \(R_j^\mu\) denote the
steady-state rejection rate of size \(j\) parcels under this policy. The
rejection cost can be written as \(J^\mu(b)=\sum_{j=1}^n \pi_j R_j^\mu\). Using the decomposition \(\pi_j=\sum_{i=1}^j(\pi_i-\pi_{i-1})\), with \(\pi_0:=0\),
we obtain
\[
J^\mu(b)
=
\sum_{i=1}^n
(\pi_i-\pi_{i-1})
\sum_{j=i}^n R_j^\mu.
\]
Thus, for each \(i\), it is enough to derive a lower bound on the total rejection rate
of sizes \(i,\ldots,n\).

Fix a size index \(i\). Parcels of sizes \(i,\ldots,n\) can occupy only
lockers of sizes \(i,\ldots,n\). Hence, at any time, the total number of
accepted parcels from sizes \(i,\ldots,n\) is at most
\(B_i(b)=\sum_{k=i}^n b_k\).
Now consider an idealized relaxation in which all parcels of sizes
\(i,\ldots,n\) are pooled into a single size with aggregate arrival rate
\(\Lambda_i=\sum_{k=i}^n\lambda_k\)
and have exclusive access to \(B_i(b)\) identical lockers. This relaxation
removes all size-compatibility restrictions inside the tail set and removes any interference from smaller parcels of sizes \(1,\ldots,i-1\).
Therefore, any admission sequence feasible in the original system for
sizes \(i,\ldots,n\) is feasible in this relaxed pooled system.

In the relaxed single-size system, the policy that admits every arriving
parcel whenever a locker is available minimizes the steady-state rejection
rate. Indeed, all parcels are statistically identical, rejection has a unit
cost, and accepted parcels have the same memoryless pickup probability
\(p\). Idling an available locker cannot increase future capacity in a way
that distinguishes future parcels from the current one. Equivalently, under a monotone coupling, the accept-when-available policy
stochastically maximizes the number of occupied lockers in every period,
and therefore maximizes long-run throughput. Because the arrival rate is fixed,
maximizing throughput is equivalent to minimizing the rejection rate.

The minimum rejection rate in this relaxed pooled system is
\(r(\Lambda_i,B_i(b))\). Since the original system is more constrained than
this relaxation, the total rejection rate of sizes \(i,\ldots,n\) in the
original system under policy \(\mu\) must satisfy
\[
\sum_{j=i}^n R_j^\mu
\ge
r(\Lambda_i,B_i(b)).
\]
Substituting this bound into the penalty decomposition gives
\[
J^\mu(b)
=
\sum_{i=1}^n
(\pi_i-\pi_{i-1})
\sum_{j=i}^n R_j^\mu
\ge
\sum_{i=1}^n
(\pi_i-\pi_{i-1})
r(\Lambda_i,B_i(b)).
\]
This proves the bound for every admissible stationary policy \(\mu\).
Applying it to the AA policy gives the fixed-policy bound, and
taking the infimum over \(\mu\) gives \(g^*(b)\ge \underline J(b)\).
\end{proof}

The tail-pooled bound ignores much of the congestion structure. A stronger, configuration-specific relaxation is
obtained by allowing preemption: Parcels already in the
lockers may be evicted at the next replenishment epoch if newly arriving
larger-size parcels need their compatible capacity. An evicted parcel incurs
the same penalty as a parcel rejected upon arrival. This is not a feasible
operating policy in the original model, because admitted parcels cannot be removed
before customer pickup, but it is a valid relaxation of the admission-control policy
problem. The relaxation preserves the stochastic pickup process, while removing only the irreversibility of past admission
decisions.

The corresponding DTMC is defined on the same state space
\(\mathcal S_n\). Given a state \(x\) and arrivals \(a\), let
\(N_i=x_i+a_i\) be the number of size-\(i\) parcels available for retention.
The preemptive relaxation keeps the highest-penalty feasible collection of
parcels. In tail notation, set \(T_{n+1}=0\) and define \(T_i^{P}(x,a)=\min\{B_i,T_{i+1}^{P}(x,a)+N_i\}\), \(i=n,n-1,\ldots,1\), with \(y_i^{P}(x,a)=T_i^{P}(x,a)-T_{i+1}^{P}(x,a)\). The one-period relaxed cost is \(c^{P}(x,a)=\sum_{i=1}^n \pi_i(N_i-y_i^{P}(x,a))\), which counts both rejected arrivals and evicted parcels. Conditional on
\(y^{P}\), pickups are independent binomials as in the original model, and the
next state is \(x'=y^{P}-D\). Let \(\underline J^{P}(b)\) be the stationary
average cost of this finite DTMC.

\begin{lemma}
\label{lem:preemptive-greedy}
For every state \(x\) and arrival vector \(a\), the vector \(y^{P}(x,a)\) defined by the tail recursion above minimizes the relaxed one-period rejection cost over \(\mathcal Y^{P}(x,a)\). Moreover, using this greedy retention policy at every replenishment epoch is optimal for the preemptive relaxation.
\end{lemma}

\begin{proof}
Fix \(x\) and \(a\), and write \(N_i=x_i+a_i\). Minimizing rejected-penalty cost
is equivalent to maximizing the retained value \(\sum_i \pi_i y_i\). For any
feasible retained vector, define the tail loads \(T_i=\sum_{k=i}^n y_k\), with
 \(T_{n+1}=0\). Then feasibility is equivalent to \(0\le T_i-T_{i+1}\le N_i\) and \(T_i\le B_i\), \(i=1,\ldots,n\). The retained value can be written as \(\sum_{i=1}^n \pi_i y_i=\pi_1T_1+\sum_{i=2}^n(\pi_i-\pi_{i-1})T_i\).
All coefficients in this expression are positive because
\(0<\pi_1<\cdots<\pi_n\). It is therefore sufficient to maximize all tail loads
component-wise. The recursion \(T_i^{P}=\min\{B_i,T_{i+1}^{P}+N_i\}\), \(i=n,\ldots,1\), does exactly this. By backward induction, any feasible vector satisfies
\(T_n\le T_n^P\). If \(T_{i+1}\le T_{i+1}^P\), then
\[
T_i\le \min\{B_i,T_{i+1}+N_i\}
\le
\min\{B_i,T_{i+1}^P+N_i\}=T_i^P.
\]
Thus \(T^P\) dominates every feasible tail-load vector component-wise, and hence
maximizes retained value. The corresponding size counts are
\(y_i^P=T_i^P-T_{i+1}^P\).

Finally, future preemption makes this state-by-state greedy decision dynamically
optimal. Retaining an additional feasible parcel cannot make a future
larger-size parcel infeasible, because the retained parcel can be evicted later
and charged its rejection penalty if such a conflict occurs. Thus, rejecting a
parcel earlier cannot improve future opportunities relative to retaining it and
possibly evicting it later. Hence, applying \(y^P(x, a)\) at every replenishment
epoch is optimal for the relaxed problem.
\end{proof}

\begin{theorem}
\label{thm:preemptive-lower-bound}
For every capacity vector \(b\), the preemptive-relaxation value lower-bounds the optimized rejection cost: \(\underline J^{P}(b)\le g^*(b)\).
\end{theorem}

\begin{proof}
Consider the relaxed action set
\[
\mathcal Y^{P}(x,a)=
\left\{
y\in\mathbb Z_+^n:
0\le y_i\le x_i+a_i,\quad
\sum_{k=i}^n y_k\le B_i,\quad i=1,\ldots,n
\right\}.
\]
This set contains the original action set \(\mathcal U(x, a)\): it drops both
the constraints that already-present parcels must remain, \(y_i\ge x_i\), and
the lower bounds requiring arrivals to be admitted into available associated size
lockers. Hence, every original admission decision is feasible in the preemptive
relaxation, with the same immediate cost and pickup transition. The
optimal average cost of the relaxed MDP is therefore no larger than \(g^*(b)\).

By Lemma~\ref{lem:preemptive-greedy}, the tail-recursion policy \(y^P(x, a)\) is
optimal for the preemptive relaxation. The finite DTMC described above is the
Markov chain induced by this optimal relaxed policy, and its stationary average
cost is \(\underline J^{P}(b)\). Therefore \(\underline J^{P}(b)\) is the
optimal value of the preemptive relaxation and is no larger than \(g^*(b)\).
\end{proof}

Computing the lower bound in Theorem~\ref{thm:tail-pooled-lower-bound} is several orders of magnitude faster than solving the MDP presented in Section~\ref{sec:nsize-model}, or even computing the expected cost under AA, because it decomposes the system into \(n\) one-dimensional DTMCs, each with only \(B_i+1\) states. The preemptive bound in Theorem~\ref{thm:preemptive-lower-bound} is typically more expensive because its DTMC uses the full state space, but it preserves the actual congestion structure and may therefore be tighter. The screening order is therefore sequential: first apply the inexpensive tail-pooled bound, then compute the preemptive-relaxation bound only for configurations that survive the first screen, and solve the exact MDP only for configurations that survive both. The remaining missing piece is a good heuristic for finding the first incumbent solution, and the following local search is effective for this purpose. The search is initialized at
\(b_i^{0}=\lceil \lambda_i/p\rceil\), \(i=1,\dots,n\), which is the nominal number of size-\(i\) lockers needed to match the mean number of size-\(i\) parcels in the system under a no-spillover approximation. Let \(e_i\) denote the \(i\)th unit vector. From any current capacity vector \(b\), the heuristic evaluates the add/drop and exchange neighborhood
\[
\mathcal N(b)=\{b+e_i,\; b-e_i:\ i=1,\dots,n\}
\cup
\{b+e_i-e_j:\ i,j=1,\dots,n,\ i\ne j\},
\]
and moves to the neighbor with the largest decrease in \(C^*(b)\). Each exact evaluation usually solves the MDP for the current capacity vector. The only exception is when the sufficient condition of Theorem~\ref{thm:nsize-always-accept} certifies that AA is optimal for that vector; in that case, the fixed-policy AA DTMC gives the same value \(g^*(b)\). The procedure stops when no neighbor improves the objective.

The two evaluation models can be combined sequentially. One can first run the local search with the fixed-policy DTMC under AA, which is much cheaper than solving the full MDP because no Bellman optimization is required. The resulting local optimum is then computed by value iteration and used as the starting point for a second local search in which each neighbor is evaluated under the optimized admission policy. This two-stage variant uses the DTMC solution as a warm start for the more expensive MDP search. When the second stage switches to value iteration, however, the DTMC objective values computed in the first stage cannot be reused as exact MDP evaluations; only capacity vectors evaluated under the active value iteration objective may be cached for the subsequent bound-and-enumerate step. The two-stage procedure is therefore an incumbent-generation heuristic until followed by the bound-and-enumerate certification, but it is attractive when the AA surface is a close approximation to the optimized surface, as suggested by the numerical evidence in Sections~\ref{sec:numerical-evidence} and~\ref{sec:locker-assortment}.

Algorithm~\ref{alg:assortment-optimization} combines this incumbent-search step with the simplex restriction and the two relaxation lower bounds above. We refer to it as a bound-and-enumerate algorithm: it first obtains an incumbent via local search, then enumerates the incumbent simplex, discarding configurations that cannot improve it using lower bounds. The cheap tail-pooled bound is evaluated first, and the stronger preemptive bound is evaluated only for candidates that remain potentially improving. It is therefore an exact optimization method for the full MDP locker-assortment objective \(C^*(b)\). The sufficient condition of Theorem~\ref{thm:nsize-always-accept} is used only as an evaluation shortcut: when it certifies a candidate, the AA DTMC computes the exact MDP value for that candidate; otherwise, value iteration is used. The candidate configurations are sorted in increasing order of
\(B_1(b)=\sum_{i=1}^n b_i\), the total number of lockers, with facility cost as a secondary key. This tends
to evaluate smaller exact Markov chains first while avoiding rigid layer-by-layer
batches. Whenever the incumbent improves, the upper bound \(U\) decreases, and
the queue is filtered to remove configurations that have left the
incumbent simplex.

\begin{algorithm}[htbp]
\caption{Bound-and-enumerate locker-assortment optimization}
\label{alg:assortment-optimization}
\scriptsize
\begin{algorithmic}[1]
\Require Arrival rates \(\lambda\), pickup probability \(p\), penalties \(\pi\), locker costs \(f\), initial capacity vector \(b^0\)
\Ensure Optimal capacity vector \(b^\star\) and objective value \(U\) for \(C^*(b)\)
\State Run exchange-neighborhood local search from \(b^0\), using exact evaluations of \(C^*(b)\)
\State Let \(b^\star\) be the local-search solution and set \(U\gets C^*(b^\star)\)
\State Let \(\mathcal E\) be the set of capacity vectors already evaluated exactly during local search
\State Build \(\mathcal Q\gets\{b\in\mathbb Z_+^n:\sum_{i=1}^n f_i b_i<U,\ b\notin\mathcal E\}\), sorted by \(\left(\sum_i b_i,\sum_i f_i b_i\right)\)
\For{each unevaluated \(b\in\mathcal Q\), in sorted order}
    \If{\(\sum_{i=1}^n f_i b_i<U\)} \Comment{May fail after an incumbent update}
        \State \(L_0(b)\gets \underline J(b)+\sum_{i=1}^n f_i b_i\) \Comment{Cheap tail-pooled screen}
        \If{\(L_0(b)<U\)}
            \State \(L_P(b)\gets \underline J^{P}(b)+\sum_{i=1}^n f_i b_i\) \Comment{Stronger preemptive screen}
            \If{\(L_P(b)<U\)}
                \If{Theorem~\ref{thm:nsize-always-accept} certifies AA for \(b\)}
                    \State Evaluate \(g^*(b)=g^{AA}(b)\) by the fixed-policy DTMC
                \Else
                    \State Evaluate \(g^*(b)\) by solving the MDP with value iteration
                \EndIf
                \State Set \(C^*(b)\gets g^*(b)+\sum_{i=1}^n f_i b_i\)
                \If{\(C^*(b)<U\)}
                    \State \(b^\star\gets b\), \(U\gets C^*(b)\)
                    \State Remove from \(\mathcal Q\) every remaining \(b'\) with \(\sum_i f_i b'_i\ge U\)
                \EndIf
            \EndIf
        \EndIf
    \EndIf
\EndFor
\State \Return \(b^\star\), \(U\)
\end{algorithmic}
\normalsize
\end{algorithm}

The implementation parallelizes the computationally expensive exact-evaluation
steps without changing the logic of Algorithm~\ref{alg:assortment-optimization}.
During the local-search phase, all feasible neighbors of the current capacity
vector are evaluated concurrently, and the algorithm moves only after all neighbor evaluations have returned, preserving the best-improvement policy. During the bound-and-enumerate phase, the sorted queue provides a common pool of candidate configurations. The screening stages can be parallelized. The
tail-pooled bound is usually inexpensive, but the preemptive-relaxation bound in
Theorem~\ref{thm:preemptive-lower-bound} may require solving a relatively large
DTMC and is therefore a target for concurrent evaluation. In the
parallel implementation, batches of candidates are first screened by the
tail-pooled bound; candidates that survive are then dispatched to independent
workers for preemptive-bound evaluation; candidates that survive both screens are
dispatched to the exact evaluator. Whenever a worker returns a better incumbent,
the queue is filtered against the new value of \(U\). The parallelization
therefore affects only wall-clock time. The incumbent, the upper bound, and the
final certificate are computed from the same lower bounds and exact evaluations
as in the serial version. In our implementation the Python driver uses
concurrent workers to launch independent C++ evaluator calls, so the parallelism
is at the level of candidate designs rather than inside the Markov chain solvers
themselves.

Upon termination, \(b^\star\) is globally optimal for the full MDP locker-assortment objective \(C^*(b)\). Optimality follows from the incumbent simplex and the valid lower bounds in Theorems~\ref{thm:tail-pooled-lower-bound} and~\ref{thm:preemptive-lower-bound}. Any design not included in the initial queue has facility cost at least the initial incumbent value and therefore cannot improve it. Any design removed after an incumbent update has facility cost at least the new incumbent value and cannot improve it. Every remaining candidate in the incumbent simplex has either been pruned by a valid lower bound or evaluated exactly. The evaluation is exact whether it is obtained by value iteration or by the AA DTMC after Theorem~\ref{thm:nsize-always-accept} has certified that AA is optimal for that particular capacity vector.

The optimization experiments are conducted on relatively small instances, where discreteness effects are strongest and local-search failure should be easiest to detect. In such systems, adding or removing one locker can substantially change the state space and the blocking probabilities.
As the system scales, the marginal effect of a single locker becomes more
incremental, and the locker-assortment surface is expected to become smoother.
Thus, although the small-instance results do not prove that local search is globally optimal for larger instances, they provide numerical evidence that
exchange-neighborhood local search performs reliably in the tested regimes.  In high-demand
instances, the full bound-and-enumerate method no longer provides a formal optimality certificate. In those cases,
the local-search component should be viewed as a scalable heuristic. The
combination of small-instance exact optimization, the observed smoothness of
the assortment surface, and the diminishing relative effect of individual
locker changes makes local search a practical approach. In all tests
reported below, local search already finds the global optimum, so the
enumeration phase verifies the incumbent but does not improve it.

\subsection{Numerical Validation and Large-Scale Heuristics}\label{sec:assortment-numerics}
For each instance, the objective is
\(C^*(b)=g^*(b)+\sum_{i=1}^n f_i b_i\), where \(g^*(b)\) is the optimized steady-state rejection cost and \(f_i>0\) is the unit facility cost of a size-\(i\) locker. The numerical experiment applies Algorithm~\ref{alg:assortment-optimization} to a demand grid based on the load regimes in Table~\ref{tab:nsize-experiment-design}, but the assortment vector is now the decision variable. To define the grid without reference to a candidate assortment, each scale group uses the corresponding base capacity vector from Table~\ref{tab:nsize-experiment-design}. For a normalized load profile \(\rho\), the offered-load vector is \(L_i=\rho_i b_i^{\mathrm{base}}\), and the Poisson arrival rate is \(\lambda_i=p L_i\). The supplement lists the resulting numerical rates for reproducibility.

Table~\ref{tab:assortment-search-experiment-design} summarizes the experimental design. For each generated instance, we run local search from \(b_i^0=\lceil\lambda_i/p\rceil\) using the add/drop plus exchange neighborhood. Exact objective evaluations solve the full MDP by value iteration unless Theorem~\ref{thm:nsize-always-accept} certifies AA for the tested capacity vector, in which case the fixed-policy DTMC gives the same value at lower computational cost. The enumeration phase uses valid relaxation lower bounds to prune configurations before exact evaluation. The optional AA DTMC search is used only to warm-start the value iteration local search; it does not enter the certificate unless the theorem certifies exactness for the specific vector being evaluated.
\begin{table}[htbp]
\centering
\caption{Locker-assortment experiment design}
\label{tab:assortment-search-experiment-design}
\scriptsize
\begin{tabular}{@{}>{\raggedright\arraybackslash}p{0.28\textwidth}>{\raggedright\arraybackslash}p{0.31\textwidth}>{\raggedright\arraybackslash}p{0.31\textwidth}@{}}
\toprule
& $n=3$ & $n=4$  \\
\midrule
Penalty vectors & \((1,2,4)\), \((1,2,8)\), \((1,4,16)\) & \((1,2,4,8)\), \((1,2,4,16)\), \((1,3,9,27)\) \\
Facility-cost vector \(f\) & \((0.15,0.35,0.80)\) & \((0.15,0.35,0.80,1.80)\) \\
Search initialization & \multicolumn{2}{c}{\(b_i^0=\lceil\lambda_i/p\rceil\)} \\
Reported neighborhood & \multicolumn{2}{c}{Add/drop plus exchange moves} \\
Evaluation model & \multicolumn{2}{c}{Full MDP by value iteration, with certified AA DTMC shortcut} \\
Optimization policy & \multicolumn{2}{c}{Algorithm~\ref{alg:assortment-optimization} with relaxation lower-bound pruning} \\
Total generated instances & \(135\) & \(135\) \\
\bottomrule
\end{tabular}
\normalsize
\end{table}

For each instance, we record the final capacity vector, objective value, whether local search found the global optimum, the number of exact MDP evaluations, the number of lower-bound prunings, and wall-clock time split between local search and certification. Table~\ref{tab:assortment-certification-results} reports the exact certification results for the small, medium, and large groups. Across all \(270\) instances, the local-search incumbent is globally optimal. The reported mean AA gap is the gap between the best AA-DTMC assortment found in the first heuristic phase and the certified optimum; Table~\ref{tab:assortment-simulation-heuristic} later reports the corresponding gaps for assortments selected by the simulation heuristic. Thus, in these tests, the enumeration phase establishes optimality but never improves the local-search solution.

\begin{table}[htbp]
\centering
\caption{Exact locker-assortment certification results}
\label{tab:assortment-certification-results}
\tiny
\setlength{\tabcolsep}{3pt}
\resizebox{\textwidth}{!}{%
\begin{tabular}{@{}lrrrrrrrrrrr@{}}
\toprule
Scale &
\begin{tabular}[c]{@{}c@{}}Local search\\optimal\end{tabular} &
\begin{tabular}[c]{@{}c@{}}AA\\optimal\end{tabular} &
\begin{tabular}[c]{@{}c@{}}Mean\\AA gap\end{tabular} &
\begin{tabular}[c]{@{}c@{}}Mean\\candidates\end{tabular} &
\begin{tabular}[c]{@{}c@{}}Pruned by\\bounds\end{tabular} &
\multicolumn{2}{c}{Exact evals.} &
\multicolumn{2}{c}{Search time} &
\multicolumn{2}{c}{Cert. time} \\
\cmidrule(lr){7-8}\cmidrule(lr){9-10}\cmidrule(l){11-12}
& & & & & &
\multicolumn{1}{c}{mean} &
\multicolumn{1}{c}{max} &
\multicolumn{1}{c}{mean} &
\multicolumn{1}{c}{max} &
\multicolumn{1}{c}{mean} &
\multicolumn{1}{c}{max} \\
\midrule
Small & 90/90 & 72/90 & 0.037\% & 6{,}673 & 99.64\% & 23.8 & 259 & 16.9 s & 5.0 m & 16.0 s & 4.6 m \\
Medium & 90/90 & 77/90 & 0.030\% & 10{,}208 & 99.75\% & 25.2 & 283 & 59.6 s & 17.0 m & 50.3 s & 13.5 m \\
Large & 90/90 & 80/90 & 0.006\% & 19{,}999 & 99.88\% & 25.0 & 415 & 8.5 m & 1.9 h & 6.6 m & 2.4 h \\
\bottomrule
\end{tabular}
}
\normalsize
\end{table}

Table~\ref{tab:assortment-certification-results} highlights two computational patterns. First, the exchange-neighborhood local search finds the global optimum in all certified instances; the first incumbent found by local search is already optimal. Second, the relaxation bounds are essential for certification. The incumbent simplex contains thousands to tens of thousands of candidate assortments, but the tail-pooled and preemptive lower bounds prune more than \(99.6\%\) of them in every scale group. Consequently, the algorithm solves only about \(24\) exact MDPs per instance on average during certification. The large group is computationally heavier because exact value iteration calls become more expensive, but the method still certifies all \(90\) large instances; the most difficult one required about \(1.9\) hours to solve and about \(2.4\) hours to certify.

For larger systems, exact value iteration and bound-and-enumerate certification may become computationally unavailable. 
In that regime, the local-search component can still be used as a scalable heuristic, although global optimality is no longer certified. 
One practical shortcut is to replace exact evaluation of \(g^*(b)\) by discrete-event simulation of the AA operating policy. 
This simulation oracle estimates \(g^{AA}(b)\), which is close to \(g^*(b)\) in the larger service-point instances tested above, and can be much faster than value iteration or exact DTMC evaluation under AA. It can therefore be used for scalable assortment search when an optimality certificate is not required.
A parallel simulation of the preemptive relaxation in Theorem~\ref{thm:preemptive-lower-bound} can provide a lower-bound estimate against which simulation-optimized assortments can be compared when exact DTMC or value-iteration evaluation is unavailable.

Table~\ref{tab:assortment-simulation-heuristic} summarizes the simulation-only local-search runs. The first three scale groups can be compared with the exact value iteration certificates in Table~\ref{tab:assortment-certification-results}. The simulation heuristic returned the optimal capacity vector obtained by value iteration in \(79\), \(85\), and \(87\) of the \(90\) small, medium, and large instances, respectively. For these three groups, we also re-evaluated each simulation-selected assortment under the AA DTMC and compared \(C^{AA}(\hat b^{\mathrm{sim}})\) with the certified full-MDP optimum \(C^*(b^*)\). The mean exact-AA gaps are \(0.037\%\), \(0.031\%\), and \(0.006\%\), and the largest gaps are \(0.746\%\), \(0.646\%\), and \(0.230\%\), respectively. These gaps are deterministic diagnostics rather than simulation-noise artifacts; if AA is not optimal at a selected assortment, they can overstate the true full-MDP assortment gap. For the larger groups, no exact certificate is available, but the computation remains fast: even in the 4XL group, where the largest returned assortment has \(709\) lockers, the median run time is \(76.5\) seconds and the maximum is \(3.4\) minutes. 

\begin{table}[htbp]
\centering
\caption{Simulation-only local search for larger locker-assortment instances}
\label{tab:assortment-simulation-heuristic}
\scriptsize
\resizebox{\textwidth}{!}{%
\begin{tabular}{@{}lrrrrrrrrr@{}}
\toprule
Scale & Instances &
\begin{tabular}[c]{@{}c@{}}Mean total\\capacity\end{tabular} &
\begin{tabular}[c]{@{}c@{}}Max total\\capacity\end{tabular} &
\begin{tabular}[c]{@{}c@{}}Mean\\evaluations\end{tabular} &
\begin{tabular}[c]{@{}c@{}}Mean\\steps\end{tabular} &
\begin{tabular}[c]{@{}c@{}}Median\\time\end{tabular} &
\begin{tabular}[c]{@{}c@{}}Match with\\ optimum\end{tabular} &
\begin{tabular}[c]{@{}c@{}}Mean \\AA gap\end{tabular} &
\begin{tabular}[c]{@{}c@{}}Max \\AA gap\end{tabular} \\
\midrule
Small & 90 & 17.8 & 25 & 41.4 & 1.4 & 1.8 s & 79 / 90 & 0.037\% & 0.746\% \\
Medium & 90 & 23.2 & 32 & 44.3 & 1.8 & 2.2 s & 85 / 90 & 0.031\% & 0.646\% \\
Large & 90 & 32.2 & 48 & 47.7 & 2.2 & 2.9 s & 87 / 90 & 0.006\% & 0.230\% \\
XL & 90 & 64.5 & 92 & 61.0 & 3.9 & 5.7 s & -- & -- & -- \\
XXL & 90 & 129.2 & 182 & 82.4 & 6.2 & 13.2 s & -- & -- & -- \\
3XL & 90 & 258.8 & 358 & 109.5 & 9.5 & 31.1 s & -- & -- & -- \\
4XL & 90 & 517.8 & 709 & 146.4 & 13.3 & 76.5 s & -- & -- & -- \\
\bottomrule
\end{tabular}
}
\normalsize
\end{table}

The evidence is consistent with the intuition that, as the system grows large, the AA policy is more likely to be optimal or very close to optimal, and that the $g^*(b)$ surface behaves well and lends itself to local search. Therefore, for a large-scale system with high demand, simulation-based assortment search can yield sound solutions quickly.

\section{Conclusion}\label{sec:conclusion}

This paper studies admission control and locker assortment for multi-size automated parcel lockers where smaller parcels can use larger lockers, but doing so may block future larger parcels with higher rejection penalties. We formulate this trade-off as a finite-state average-cost MDP, solve validation instances by value iteration, and identify a conservative sufficient condition under which the AA policy is optimal.
Across broad two-, three-, and four-size experiments, AA is optimal or near optimal in all tested cases. The AA policy is not always optimal when parcel holding times are long, but even then, the observed gaps between AA and the optimal policy are negligible. This makes AA a practical baseline for service-point operation and allows the system to be evaluated using a fixed-policy DTMC or estimated via simulation rather than a full MDP.
The same framework supports the locker-assortment problem. Although the optimized rejection cost surface is not globally discrete convex, the bound-and-enumerate method certifies optimal assortments for all validation instances for which exact certification is computationally tractable, and exchange-neighborhood local search finds the certified optimum in every such instance. For larger systems, simulation-based local search provides a scalable heuristic that remains consistent with the exact evidence on smaller systems.
The model assumes that parcels can be relocated from larger to smaller lockers without operational cost; in practice, this additional courier task should be measured and incorporated into both the admission and assortment models. The memoryless pickup assumption is also a simplification: if pickup probabilities depend on parcel age, a fully optimal policy could use the seniority profile of parcels in the station. Developing such age-aware policies would require a richer state description and is a natural direction for future work, especially with approximate dynamic programming or reinforcement-learning methods for larger systems.

\clearpage
\appendix
\begin{center}{\Large\textbf{Supplementary Material}}\end{center}
\medskip

\section{Robustness to the Pickup-Time Assumption}\label{sec:pickup-robustness}
The main model assumes that each parcel is picked up independently in each cycle with probability $p$. This keeps the state Markovian in the parcel-count vector. Actual locker systems usually impose a finite pickup window, so we test whether the rejection-cost estimates are sensitive to the memoryless assumption.

The empirical evidence supports using short, right-skewed holding-time distributions. In a residential-building locker study in Seattle, \citet{RanjbariEtAl2023RightSize} report a median time-to-pickup of $4.2$ hours and a mean of $12.5$ hours; $84\%$ of packages were collected within $24$ hours, and $99\%$ within one week. In an Australia Post dataset covering 51 lockers and more than 867,000 pickup records, \citet{LeungLachapelleBurke2023} find that parcel dwell times stabilized around $15$--$16$ hours, with longer weekend holding times. In the Amazon Locker setting, \citet{SethuramanEtAl2024} estimate dwell-time probabilities over $0,1,\dots,6$ days by shipping option and delivery day.
The pickup probabilities used in the main experiments are therefore conservative stress regimes: $p=0.5$ gives a mean holding time of two replenishment cycles, while $p=0.33$ and $p=0.25$ represent slower pickup. The memoryless assumption is still an approximation. To assess its effect, we compare it with finite discrete holding-time distributions that have the same mean but concentrate pickup mass in the first few days.

Let $F=(\eta_1,\dots,\eta_M)$ be a finite pickup-time distribution, where $\eta_t$ is the probability of pickup after $t$ cycles, and let \(\bar T=\sum_{t=1}^M t\eta_t\). We compare two AA systems with the same capacities, arrival rates, and penalties. The first uses $F$ and is evaluated by long-run simulation. The second uses the matched memoryless model with \(p=1/\bar T\) and is evaluated exactly by the AA Markov chain. The comparison statistic is \((g^{FD}-g^M)/g^M\), where $g^{FD}$ is the simulated average rejection cost under $F$ and $g^M$ is the exact memoryless value.

The experiment uses the three- and four-size validation regimes of the main paper, with the capacity grid broadened to four capacity vectors for each size count (the main-paper capacity vectors augmented by one additional vector per size count), combined with five normalized-load profiles, three penalty ladders, and four pickup-time distributions, giving $480$ instances. Each finite-distribution instance is simulated for $1{,}000{,}000$ cycles, with $9{,}900$ measured blocks of 100 cycles after warm-up. Table~\ref{tab:pickup-robustness} reports signed and absolute relative differences between the finite-distribution simulation and the exact memoryless estimate.

\begin{table}[htbp]
\centering
\caption{Robustness to the memoryless pickup assumption}
\label{tab:pickup-robustness}
\small
\begin{tabular}{lcccccc}
\toprule
Pickup distribution & $\bar T$ & Inst. & Mean diff. & Mean abs.\ diff. & Mean rel.\ half-CI & Max abs.\ diff. \\
\midrule
Overall & -- & 480 & $-0.098\%$ & $0.168\%$ & $0.317\%$ & $0.900\%$ \\
$(0.7,0.2,0.1)$ & 1.40 & 120 & $-0.001\%$ & $0.124\%$ & $0.288\%$ & $0.497\%$ \\
$(0.5,0.3,0.2)$ & 1.70 & 120 & $-0.102\%$ & $0.148\%$ & $0.305\%$ & $0.587\%$ \\
$(0.3,0.3,0.4)$ & 2.10 & 120 & $-0.227\%$ & $0.256\%$ & $0.325\%$ & $0.900\%$ \\
$(0.35,0.25,0.2,0.1,0.1)$ & 2.35 & 120 & $-0.061\%$ & $0.145\%$ & $0.348\%$ & $0.671\%$ \\
\bottomrule
\end{tabular}
\normalsize
\end{table}

The memoryless model slightly overestimates rejection costs on average: the mean signed relative difference is $-0.098\%$, and the finite-distribution cost is lower in $320$ of the $480$ instances. The delayed three-point distribution $(0.3,0.3,0.4)$ is the most challenging case, with the largest signed and absolute relative deviations. Even there, the maximum absolute relative difference is below $1\%$.

Thus, across the tested three- and four-size instances, replacing a finite pickup-time distribution by a geometric pickup process with the same mean changes the AA rejection cost by less than $1\%$ in every case, with median absolute error about $0.13\%$. This does not prove that AA remains optimal among all age-aware policies: once pickups are not memoryless, parcel age becomes part of the true state. The conclusion is limited to operating-cost stability under the non-memoryless pickup distributions tested here. Similar robustness results for the single-service-point case are reported by \citet{raviv2023service}.

\section{Reproducibility and Computational Tools}\label{sec:reproducibility}
The numerical results were produced with Python and C++. Python handles experiment orchestration, data processing, workbook generation, plotting, and parallel scheduling of independent design evaluations in the bound-and-enumerate assortment algorithm. C++ is used for the computational kernels: relative value iteration, exact Markov-chain evaluation under fixed policies, and simulation of finite pickup-time distributions.

The replication scripts used Python~3.11.8. The C++ kernels were compiled from the Python wrappers with \texttt{-O3}; the newer general-purpose kernels use \texttt{-std=c++17}. The larger runs were executed on Linux workstations with an AMD Ryzen~9~5950X CPU and the system \texttt{g++} compiler. The same code was also tested on macOS with \texttt{clang}. The C++ solvers are single-threaded; all parallelism comes from the Python driver, which launches independent evaluator calls for different candidate designs.

The scripts, source files, and generated experiment outputs are maintained under Git version control so that the reported tables and figures can be traced to the corresponding runs. The repository is available at \url{https://github.com/tal69/sp-multi-size}.

\section{Locker-Assortment Experiment Demand Grid}\label{app:assortment-search-grid}

Table~\ref{tab:assortment-search-arrivals} reports the arrival rates used in the locker-assortment experiment. For each scale group, the offered-load vector is \(\lambda/p\); multiplying by the pickup probability \(p\) gives the Poisson arrival-rate vector \(\lambda\). The last three columns give the arrival-rate vectors for the three pickup regimes used in the experiment.

\begin{table}[htbp]
\centering
\caption{Arrival-rate and offered-load grid for the locker-assortment experiment}
\label{tab:assortment-search-arrivals}
\scriptsize
\resizebox{\textwidth}{!}{%
\begin{tabular}{@{}cccccc@{}}
\toprule
Sizes & Scale & Offered load \(\lambda/p\) & \(\lambda\) for \(p=0.25\) & \(\lambda\) for \(p=0.33\) & \(\lambda\) for \(p=0.50\) \\
\midrule
3 & Small & \((9.60,3.20,1.60)\) & \((2.40,0.80,0.40)\) & \((3.17,1.06,0.53)\) & \((4.80,1.60,0.80)\) \\
3 & Small & \((12.00,4.00,2.00)\) & \((3.00,1.00,0.50)\) & \((3.96,1.32,0.66)\) & \((6.00,2.00,1.00)\) \\
3 & Small & \((14.40,4.80,2.40)\) & \((3.60,1.20,0.60)\) & \((4.75,1.58,0.79)\) & \((7.20,2.40,1.20)\) \\
3 & Small & \((14.40,4.00,1.60)\) & \((3.60,1.00,0.40)\) & \((4.75,1.32,0.53)\) & \((7.20,2.00,0.80)\) \\
3 & Small & \((9.60,4.00,2.40)\) & \((2.40,1.00,0.60)\) & \((3.17,1.32,0.79)\) & \((4.80,2.00,1.20)\) \\
\addlinespace
3 & Medium & \((12.80,4.00,2.40)\) & \((3.20,1.00,0.60)\) & \((4.22,1.32,0.79)\) & \((6.40,2.00,1.20)\) \\
3 & Medium & \((16.00,5.00,3.00)\) & \((4.00,1.25,0.75)\) & \((5.28,1.65,0.99)\) & \((8.00,2.50,1.50)\) \\
3 & Medium & \((19.20,6.00,3.60)\) & \((4.80,1.50,0.90)\) & \((6.34,1.98,1.19)\) & \((9.60,3.00,1.80)\) \\
3 & Medium & \((19.20,5.00,2.40)\) & \((4.80,1.25,0.60)\) & \((6.34,1.65,0.79)\) & \((9.60,2.50,1.20)\) \\
3 & Medium & \((12.80,5.00,3.60)\) & \((3.20,1.25,0.90)\) & \((4.22,1.65,1.19)\) & \((6.40,2.50,1.80)\) \\
\addlinespace
3 & Large & \((19.20,6.40,3.20)\) & \((4.80,1.60,0.80)\) & \((6.34,2.11,1.06)\) & \((9.60,3.20,1.60)\) \\
3 & Large & \((24.00,8.00,4.00)\) & \((6.00,2.00,1.00)\) & \((7.92,2.64,1.32)\) & \((12.00,4.00,2.00)\) \\
3 & Large & \((28.80,9.60,4.80)\) & \((7.20,2.40,1.20)\) & \((9.50,3.17,1.58)\) & \((14.40,4.80,2.40)\) \\
3 & Large & \((28.80,8.00,3.20)\) & \((7.20,2.00,0.80)\) & \((9.50,2.64,1.06)\) & \((14.40,4.00,1.60)\) \\
3 & Large & \((19.20,8.00,4.80)\) & \((4.80,2.00,1.20)\) & \((6.34,2.64,1.58)\) & \((9.60,4.00,2.40)\) \\
\midrule
4 & Small & \((9.60,2.40,1.60,0.80)\) & \((2.40,0.60,0.40,0.20)\) & \((3.17,0.79,0.53,0.26)\) & \((4.80,1.20,0.80,0.40)\) \\
4 & Small & \((12.00,3.00,2.00,1.00)\) & \((3.00,0.75,0.50,0.25)\) & \((3.96,0.99,0.66,0.33)\) & \((6.00,1.50,1.00,0.50)\) \\
4 & Small & \((14.40,3.60,2.40,1.20)\) & \((3.60,0.90,0.60,0.30)\) & \((4.75,1.19,0.79,0.40)\) & \((7.20,1.80,1.20,0.60)\) \\
4 & Small & \((14.40,3.15,1.80,0.75)\) & \((3.60,0.79,0.45,0.19)\) & \((4.75,1.04,0.59,0.25)\) & \((7.20,1.57,0.90,0.38)\) \\
4 & Small & \((9.00,2.70,2.10,1.20)\) & \((2.25,0.68,0.53,0.30)\) & \((2.97,0.89,0.69,0.40)\) & \((4.50,1.35,1.05,0.60)\) \\
\addlinespace
4 & Medium & \((12.80,3.20,1.60,0.80)\) & \((3.20,0.80,0.40,0.20)\) & \((4.22,1.06,0.53,0.26)\) & \((6.40,1.60,0.80,0.40)\) \\
4 & Medium & \((16.00,4.00,2.00,1.00)\) & \((4.00,1.00,0.50,0.25)\) & \((5.28,1.32,0.66,0.33)\) & \((8.00,2.00,1.00,0.50)\) \\
4 & Medium & \((19.20,4.80,2.40,1.20)\) & \((4.80,1.20,0.60,0.30)\) & \((6.34,1.58,0.79,0.40)\) & \((9.60,2.40,1.20,0.60)\) \\
4 & Medium & \((19.20,4.20,1.80,0.75)\) & \((4.80,1.05,0.45,0.19)\) & \((6.34,1.39,0.59,0.25)\) & \((9.60,2.10,0.90,0.38)\) \\
4 & Medium & \((12.00,3.60,2.10,1.20)\) & \((3.00,0.90,0.53,0.30)\) & \((3.96,1.19,0.69,0.40)\) & \((6.00,1.80,1.05,0.60)\) \\
\addlinespace
4 & Large & \((16.00,4.00,2.40,0.80)\) & \((4.00,1.00,0.60,0.20)\) & \((5.28,1.32,0.79,0.26)\) & \((8.00,2.00,1.20,0.40)\) \\
4 & Large & \((20.00,5.00,3.00,1.00)\) & \((5.00,1.25,0.75,0.25)\) & \((6.60,1.65,0.99,0.33)\) & \((10.00,2.50,1.50,0.50)\) \\
4 & Large & \((24.00,6.00,3.60,1.20)\) & \((6.00,1.50,0.90,0.30)\) & \((7.92,1.98,1.19,0.40)\) & \((12.00,3.00,1.80,0.60)\) \\
4 & Large & \((24.00,5.25,2.70,0.75)\) & \((6.00,1.31,0.68,0.19)\) & \((7.92,1.73,0.89,0.25)\) & \((12.00,2.62,1.35,0.38)\) \\
4 & Large & \((15.00,4.50,3.15,1.20)\) & \((3.75,1.12,0.79,0.30)\) & \((4.95,1.49,1.04,0.40)\) & \((7.50,2.25,1.57,0.60)\) \\
\bottomrule
\end{tabular}
}
\normalsize
\end{table}

\clearpage

\section{Notation}\label{app:notation}

\begin{longtable}{@{}>{\raggedright\arraybackslash}p{0.22\textwidth}>{\raggedright\arraybackslash}p{0.72\textwidth}@{}}
\caption{Notation}\label{tab:supp-notation}\\
\toprule
Symbol & Meaning \\
\midrule
\endfirsthead
\toprule
Symbol & Meaning \\
\midrule
\endhead
\bottomrule
\endfoot
\multicolumn{2}{@{}l}{\textit{General multi-size model}} \\
\(n\) & Number of parcel and locker sizes, ordered from smallest \(1\) to largest \(n\). \\
\(i\) & Size index. A size-\(i\) parcel can use locker sizes \(i,i+1,\dots,n\). \\
\(b_i\) & Number of lockers of size \(i\). \\
\(B_i=\sum_{k=i}^n b_k\) & Tail capacity available to parcel sizes \(i,\dots,n\). \\
\(C_i=\sum_{k=1}^i b_k\) & Total capacity of locker sizes \(1,\dots,i\). \\
\(H_i=\sum_{k=i+1}^n b_k\) & Number of larger-size lockers above interface \(i\). \\
\(\ell_i(x)\) & Number of size-\(<i\) parcels occupying size-\(i\) lockers after relocation. \\
\(\phi_i(x)\) & Number of size-\(i\) lockers available for size-\(i\) arrivals after relocation of existing parcels. \\
\(\pi_i\) & Rejection penalty for a size-\(i\) parcel, with \(\pi_1<\cdots<\pi_n\). \\
\(p\) & Per-cycle pickup probability in the memoryless pickup model. \\
\(X_t^i\) & Number of size-\(i\) arrivals in cycle \(t\). \\
\(X_t=(X_t^1,\dots,X_t^n)\) & Arrival vector in cycle \(t\). \\
\(A\) & Generic arrival vector used in the Bellman equation. \\
\(\lambda_i\), \(\lambda=(\lambda_1,\dots,\lambda_n)\) & Mean size-\(i\) Poisson arrival rate, and the corresponding arrival-rate vector, in the numerical experiments. \\
\(x_i\) & Number of size-\(i\) parcels in the system before replenishment. \\
\(x=(x_1,\dots,x_n)\) & Pre-replenishment state vector. \\
\(\mathcal S_n\) & Feasible state space for the \(n\)-size system. \\
\(a=(a_1,\dots,a_n)\) & Realized arrival vector used in the \(n\)-size Bellman equation. \\
\(y_i\) & Number of size-\(i\) parcels retained after replenishment. \\
\(y=(y_1,\dots,y_n)\) & Post-replenishment parcel-count vector. \\
\(\mathcal Y(x,a)\) & Physically feasible set of post-replenishment vectors in state \(x\) under arrivals \(a\). \\
\(\mathcal U(x,a)\) & Admissible MDP action set, equal to the vectors in \(\mathcal Y(x,a)\) that respect the decreasing-size processing policy. \\
\(D_i\) & Number of size-\(i\) pickups after replenishment. Conditional on \(y_i\), \(D_i\sim\mathrm{Bin}(y_i,p)\). \\
\(g\) & Optimal average daily rejection cost in the average-cost MDP. \\
\(h(x)\) & Relative value, or bias, function in the average-cost optimality equation. \\
\(v^{(k)}\) & Value-iteration approximation to the bias function at iteration \(k\). \\
\(Y_i=\sum_{k=i}^n y_k\) & Post-replenishment tail load from size \(i\) upward. \\
\(\rho_i=\lambda_i/(p b_i)\) & No-spillover normalized load of size \(i\), used in the numerical experiments. \\\\
\multicolumn{2}{@{}l}{\textit{Multi-size sufficient condition}} \\
\(m\) & Number of smaller-size parcels occupying larger-size lockers after replenishment. \\
\(D_{i,m}\) & Pickup count among \(C_i+m\) smaller-size parcels when \(m\) such parcels occupy larger-size lockers across interface \(i\). \\
\(R_{i,m}\) & Number of smaller-size parcels still occupying larger-size lockers after the next relocation step, starting from \(m\) smaller-size parcels above interface \(i\). \\
\(q_{i,m}\) & One-cycle persistence probability for \(m\) smaller-size parcels crossing interface \(i\). \\
\(q_i^{\max}\) & Worst-case persistence probability at interface \(i\), used in the sufficient condition for AA optimality. \\\\
\multicolumn{2}{@{}l}{\textit{Policy evaluation, robustness, and design}} \\
\(g^{AA}\) & Average rejection cost under the AA policy. \\
\(g^*\) & Optimal average rejection cost from value iteration. \\
\(F=(\eta_1,\dots,\eta_M)\) & Finite pickup-time distribution used in the robustness experiment. \\
\(\bar T=\sum_{t=1}^M t\eta_t\) & Mean holding time under a finite pickup-time distribution. \\
\(g^{FD}\) & Simulated average rejection cost under a finite pickup-time distribution. \\
\(g^M\) & Exact average rejection cost under the matched memoryless pickup model. \\
\(f_i\) & Facility cost per locker of size \(i\) in the locker-assortment section. \\
\(b^0\) & Initial capacity vector used to start local search in the locker-assortment algorithm. \\
\(C^*(b)\) & Locker-assortment objective: optimized rejection cost plus facility cost. \\
\end{longtable}

\bibliographystyle{plainnat}
\bibliography{ref}
\end{document}